\documentclass[11pt]{article}
\usepackage[a4paper]{geometry}
\usepackage{amsthm}
\usepackage{mathtools}
\usepackage{mathrsfs}
\usepackage{graphicx}
\usepackage{amssymb}
\usepackage{xtab}
\usepackage{xcolor}
\usepackage{bbm}
\usepackage{enumitem}
\usepackage{url}
\usepackage{bbm}
\usepackage{tikz}
\usetikzlibrary{matrix,arrows}

 \newcommand{\Z}{\ensuremath{\mathbb{Z}}}
 \newcommand{\R}{\ensuremath{\mathbb{R}}}

\setlist[enumerate,1]{label=(\roman*)}

\newcommand{\norm}[1]{\left\lVert#1\right\rVert}
\newcommand{\pnorm}[1]{\left|#1\right|}
\newcommand{\norminf}[1]{\left|#1\right|_\infty}
\newcommand{\eqd}{=_d}
\newcommand{\holsp}[1][\eta]{\mathcal{C}^{#1}}

\newcommand{\Ex}[2][]{\mathbb{E}_{#1}\left[#2\right]}

\newcommand{\indicd}[1]{\mathbbm{1}_{#1}}

\newcommand{\bbS}{\ensuremath{\mathbb{S}}}

\newcommand{\tf}{T}
\newcommand{\ms}{M}

\newcommand{\holder}{H\"{o}lder}
\newcommand{\map}[2]{\colon#1\rightarrow#2}
\newcommand{\Lip}{\mathrm{Lip}}

\newtheorem{prop}{Proposition}[section]
\newtheorem{lemma}[prop]{Lemma}

\newtheorem{theorem}[prop]{Theorem}

\newtheorem{remark}[prop]{Remark}

\newtheorem{defn}[prop]{Definition}

\makeatletter
\let\save@mathaccent\mathaccent
\newcommand*\if@single[3]{%
	\setbox0\hbox{${\mathaccent"0362{#1}}^H$}%
	\setbox2\hbox{${\mathaccent"0362{\kern0pt#1}}^H$}%
	\ifdim\ht0=\ht2 #3\else #2\fi
}
\newcommand*\rel@kern[1]{\kern#1\dimexpr\macc@kerna}
\newcommand*\widebar[1]{\@ifnextchar^{{\wide@bar{#1}{0}}}{\wide@bar{#1}{1}}}
\newcommand*\wide@bar[2]{\if@single{#1}{\wide@bar@{#1}{#2}{1}}{\wide@bar@{#1}{#2}{2}}}
\newcommand*\wide@bar@[3]{%
	\begingroup
	\def\mathaccent##1##2{%
		\let\mathaccent\save@mathaccent
		\if#32 \let\macc@nucleus\first@char \fi
		\setbox\z@\hbox{$\macc@style{\macc@nucleus}_{}$}%
		\setbox\tw@\hbox{$\macc@style{\macc@nucleus}{}_{}$}%
		\dimen@\wd\tw@
		\advance\dimen@-\wd\z@
		\divide\dimen@ 3
		\@tempdima\wd\tw@
		\advance\@tempdima-\scriptspace
		\divide\@tempdima 10
		\advance\dimen@-\@tempdima
		\ifdim\dimen@>\z@ \dimen@0pt\fi
		\rel@kern{0.6}\kern-\dimen@
		\if#31
		\overline{\rel@kern{-0.6}\kern\dimen@\macc@nucleus\rel@kern{0.4}\kern\dimen@}%
		\advance\dimen@0.4\dimexpr\macc@kerna
		\let\final@kern#2%
		\ifdim\dimen@<\z@ \let\final@kern1\fi
		\if\final@kern1 \kern-\dimen@\fi
		\else
		\overline{\rel@kern{-0.6}\kern\dimen@#1}%
		\fi
	}%
	\macc@depth\@ne
	\let\math@bgroup\@empty \let\math@egroup\macc@set@skewchar
	\mathsurround\z@ \frozen@everymath{\mathgroup\macc@group\relax}%
	\macc@set@skewchar\relax
	\let\mathaccentV\macc@nested@a
	\if#31
	\macc@nested@a\relax111{#1}%
	\else
	\def\gobble@till@marker##1\endmarker{}%
	\futurelet\first@char\gobble@till@marker#1\endmarker
	\ifcat\noexpand\first@char A\else
	\def\first@char{}%
	\fi
	\macc@nested@a\relax111{\first@char}%
	\fi
	\endgroup
}
\makeatother

\title{Functional Correlation Bounds and Optimal Iterated Moment Bounds for Slowly-mixing Nonuniformly Hyperbolic Maps}
\author{Nicholas Fleming-Vázquez\thanks{Mathematics Institute, University of Warwick, Coventry, CV4 7AL, UK}}
\usepackage{fullpage}
\usepackage{hyperref}
\hypersetup{
	colorlinks=true,
	linkcolor=blue,
	filecolor=magenta,      
	urlcolor=cyan,
}
\newcommand{\shat}[1]{\vphantom{#1}\smash[t]{\widehat{#1}}}
\newcommand{\stilde}[1]{\vphantom{#1}\smash[t]{\widetilde{#1}}}
\newcommand{\fcb}{Functional Correlation Bound}
\newcommand{\dholsp}{\mathcal{H}}
\newcommand{\sepdholsp}[1]{\mathcal{SH}_{#1}}
\newcommand{\dsemi}[1]{[#1]_{\mathcal{H}}}
\newcommand{\sdsemi}[2]{[#1]_{\mathcal{H},#2}}

\newcommand{\dhnorm}[1]{\norm{#1}_{\mathcal{H}}}

\numberwithin{equation}{section}

\begin{document}
	\maketitle
	\begin{abstract}
		Consider a nonuniformly hyperbolic map $ \tf\map{\ms}{\ms} $ modelled by a Young tower with tails of the form $ O(n^{-\beta}) $, $ \beta>2 $. We prove optimal moment bounds for Birkhoff sums $ \sum_{i=0}^{n-1}v\circ\tf^i $ and iterated sums $ \sum_{0\le i<j<n}v\circ\tf^i\, w\circ\tf^j $, where $ v,w\map{\ms}{\R} $ are (dynamically)~\holder{} observables. Previously iterated moment bounds were only known for $ \beta>5$. Our method of proof is as follows; (i) prove that $ \tf $ satisfies an abstract functional correlation bound, (ii) use a weak dependence argument to show that the functional correlation bound implies moment estimates.

		Such iterated moment bounds arise when using rough path theory to prove deterministic homogenisation results. Indeed, by a recent result of Chevyrev, Friz, Korepanov, Melbourne \& Zhang we have convergence to an It\^{o} diffusion for fast-slow systems of the form
		\[ x^{(n)}_{k+1}=x_k^{(n)}+n^{-1}a(x_k^{(n)},y_k)+n^{-1/2}b(x_k^{(n)},y_k) , \quad y_{k+1}=\tf y_k \]
		in the optimal range $ \beta>2. $
	\end{abstract}
	\section{Introduction}
		Let $ \tf\map{\ms}{\ms} $ be an ergodic, measure-preserving transformation defined on a bounded metric space $ (\ms,d) $ with Borel probability measure $ \mu $.
	Consider a fast-slow system on $\R^d\times \ms$ of the form 
	\begin{equation}\label{eq:basic_fast_slow}
		x_{k+1}^{(n)}=x_k^{(n)}+n^{-1}a(x_k^{(n)},y_k)+n^{-1/2}b(x_k^{(n)},y_k),\quad y_{k+1}=Ty_k
	\end{equation}
	where the initial condition $x^{(n)}_0\equiv \xi$ is fixed and $ y_0 $ is picked randomly from $ (\ms,\mu). $ When the fast dynamics $ \tf\map{\ms}{\ms} $ is chaotic enough, it is expected that the stochastic process $ X_n $ defined by $ X_n(t)=x^{(n)}_{[nt]} $ will weakly converge to the solution of a stochastic differential equation driven by Brownian motion. This is referred to as \textit{deterministic homogenisation} and has been of great interest recently \cite{dolgopyat2004limit,mstuart2011,gottwald2013homogenization,kelly2016smooth,de2016statistical,de2018limit,chevyrev2020deterministic,korepanov2020deterministic}. See \cite{chevyrevsurvey2019} for a survey of the topic.

	In \cite{kelly2016smooth}, Kelly and Melbourne considered the special case where $a(x,y)\equiv a(x)$ and $ b(x,y)=h(x)v(y) $. By using rough path theory, they showed that deterministic homogenisation reduces to proving two statistical properties for $ \tf\map{\ms}{\ms} $. In \cite{chevyrev2020deterministic} this result was extended to general $a,b$ satisfying mild regularity assumptions. 
	
	One of the assumed statistical properties is an ``iterated weak invariance principle". In \cite{kelly2016smooth,melbourne2016note} it was shown that this property is satisfied by nonuniformly expanding/hyperbolic maps modelled by Young towers, provided that the tails of the return time decay at rate $ O(n^{-\beta}) $ for some $ \beta>2 $ (which is the optimal range for such results).
	
	The second assumed statistical property is control of ``iterated moments", which gives tightness in the rough path topology used for proving convergence. This condition has proved much more problematic. Advances in rough path theory \cite{chevyrevsurvey2019,chevyrev2020deterministic} significantly weakened the moment requirements from \cite{kelly2016smooth} and these weakened moment requirements were eventually proved for nonuniformly \emph{expanding} maps in the optimal range (i.e.\ $ \beta>2 $) in \cite{korepanov2020deterministic}.
	
	However, for nonuniformly hyperbolic maps modelled by Young towers previously it was only possible to show iterated moment bounds for $ \beta>5 $ \cite{demers2020martingale}. In this article, we extend iterated moment bounds to the optimal range $ \beta>2. $
	
	\subsection{Illustrative examples}
	Many examples of invertible dynamical systems are modelled by Young towers \cite{young1998statistical,young1999recurrence}. For example, Axiom A (uniformly hyperbolic) diffeomorphisms, Henon attractors and the finite-horizon Sinai billiard are modelled by Young towers with exponential tails, so for such systems deterministic homogenisation results follow from \cite{kelly2016smooth,kelly2017dethom}. We now give some examples of slowly-mixing nonuniformly hyperbolic dynamical systems for which it was not previously possible to show deterministic homogenisation, due to a lack of control of iterated moments. We start with an example which is easy to write down:
	\begin{itemize}[leftmargin=*]
	\item \textbf{Intermittent Baker's maps.}\quad Let $ \alpha\in(0,1). $ Define $ g\map{[0,1/2]}{[0,1]} $ by $ g(x)=x(1+2^{\alpha}x^\alpha) $. 
	The Liverani-Saussol-Vaienti map $ \bar{\tf}\map{[0,1]}{[0,1]} $, 
	\[ \bar\tf x=\begin{cases}$~$
		g(x),& x\le 1/2,\\
		2x-1,& x>1/2
	\end{cases} \] 
	is a prototypical example of a slowly-mixing nonuniformly expanding map \cite{lsvmaps}. As in~\cite[Exa.~4.1]{melbourne2016note}, consider an intermittent Baker's map $ \tf\map{\ms}{\ms} $, $ \ms=[0,1]^2 $ defined by 
	\begin{equation*}
	\tf(x_1,x_2)=\begin{cases}
		(\bar{T}x_1,g^{-1}(x_2)),&x_1\in[0,\frac{1}{2}],\ x_2\in [0,1],\\
		(\bar{T}x_1,(x_2+1)/2),& x_1\in(\frac{1}{2},1],\ x_2\in [0,1].\\
	\end{cases}
\end{equation*}
	There is a unique absolutely continuous invariant probability measure $ \mu $. The map $ \tf $ is nonuniformly hyperbolic and has a neutral fixed point at $ (0,0) $ whose influence increases with $ \alpha $. In particular, $ \tf $ is modelled by a two-sided Young tower with tails of the form $ \sim n^{-\beta} $ where $ \beta=1/\alpha $.
	
	For $ \beta>2 $ the central limit theorem (CLT) holds for all \holder{} observables. For $ \beta\le 2 $ the CLT fails for typical \holder{} observables \cite{gouezel2004central}, so it is natural to restrict to $ \beta>2 $ when considering deterministic homogenisation. By \cite{demers2020martingale} it is possible to show iterated moment bounds for $ \beta>5 $. Our results yield iterated moment bounds and hence deterministic homogenisation in the full range $ \beta>2 .$
\end{itemize}
	Dispersing billiards provide many examples of slowly-mixing nonuniformly hyperbolic maps. Markarian \cite{markarian2004etds}, Chernov and Zhang \cite{chernovzhang2005nonlin} showed how to model many examples of dispersing billiards by Young towers with polynomial tails.
	
	We give two classes of dispersing billiards for which it is now possible to show deterministic homogenisation:
\begin{itemize}[leftmargin=*,label=\textbullet]
	\item \textbf{Bunimovich flowers \cite{bunimovich_flowers}.}\quad By \cite{chernovzhang2005nonlin} the billiard map is modelled by a Young tower with tails of the form $ O(n^{-3}(\log n)^{3}). $
	\item \textbf{Dispersing billiards with vanishing curvature.}\quad In~\cite{chernov2005family} Chernov and Zhang introduced a class of billiards modelled by Young towers with tails of the form $ O((\log n)^{\beta}n^{-\beta}) $ to any prescribed value of $ \beta\in (2,\infty) $.
	\end{itemize}
\par\textbf{Notation}\quad We endow $ \R^k $ with the norm $ |y|=\sum_{i=1}^k |y_i|$. 

Let $ \eta\in(0,1] $. We say that an observable $ v\map{\ms}{\R} $ on a metric space $ (\ms,d) $ is \mbox{$ \eta $-\holder{}}, and write $ v\in \holsp(\ms) $, if $ \norm{v}_\eta=\norminf{v}+[v]_\eta<\infty, $ where $ \norminf{v}=\sup_\ms |v|$ and  $[v]_\eta=\sup_{x\ne y}|v(x)-v(y)|/d(x,y)^\eta. $ 
If $ \eta=1 $ we call $ v $ Lipschitz and write $ \Lip(v)=[v]_{1}. $ For $ 1\le p\le \infty $ we use $|\cdot|_p$ to denote the $ L^p $ norm.

The rest of this article is structured as follows. In Section~\ref{section:main_results} we state our main results. Our first main result, Theorem~\ref{thm:fcb_2_sided_yt}, is that mixing nonuniformly hyperbolic maps modelled by Young towers with polynomial tails satisfy a functional correlation bound. Our second main result, Theorem~\ref{thm:moment_bd}, is that this functional correlation bound implies control of iterated moments. 

In Section~\ref{section:yt} we recall background material on Young towers and prove Theorem~\ref{thm:fcb_2_sided_yt}. In Section~\ref{section:abstract_weak_dep} we prove that our functional correlation bound implies an elementary weak dependence condition. Finally in Section~\ref{section:moment_bounds} we use this condition to prove Theorem~\ref{thm:moment_bd}.
	\section{Main results}\label{section:main_results}
	Let $ \tf\map{\ms}{\ms} $ be a nonuniformly hyperbolic map modelled by a Young tower. We state our results for the class of dynamically \holder{} observables, noting that this includes \holder{} observables.
	We delay the definitions of Young tower and dynamically \holder{} until Section~\ref{section:yt_prereq}. Let $ \dholsp(\ms) $ denote the class of dynamically \holder{} observables on $ \ms $ and let $ \dsemi{\cdot} $ denote the dynamically \holder{} seminorm.
	\begin{defn}
		Fix an integer $ q\ge 1. $ Given a function $ G\map{\ms^{q}}{\R} $ and $ 0\le i< q$ we denote
		\[\sdsemi{G}{i} =\sup_{x_0,\dots,x_{q-1}\in \ms}\dsemi{G(x_0,\dots,x_{i-1},\cdot,x_{i+1},\dots,x_{q-1})}. \]
		We call $ G $ separately dynamically \holder{}, and write $ G\in \sepdholsp{q}(\ms)$, if $ \norminf{G}+\sum_{i=0}^{q-1} \sdsemi{G}{i}<\infty. $
	\end{defn}
	Fix $\gamma>0$. We consider dynamical systems which satisfy the following property:
\begin{defn}
	Suppose that there exists a constant $C>0$ such that for all integers $ 0\le p< q,\ 0\le n_0\le \cdots\le n_{q-1} $,
	\begin{align}
		&\bigg|\int_\ms  G(\tf^{n_0}x,\dots,\tf^{n_{q-1}}x)d\mu(x)-\int_{\ms^2} G(\tf^{n_0}x_0,\dots,\tf^{n_{p-1}}x_0,\tf^{n_p}x_1,\dots,\tf^{n_{q-1}} x_1)d\mu(x_0)d\mu(x_1)\bigg|	\nonumber\\
		&\qquad\le C(n_p-n_{p-1})^{-\gamma}\biggl(\norminf{G}+\sum_{i=0}^{q-1} \sdsemi{G}{i}\biggr)	\label{eq:fcb_bd}	
	\end{align}
	for all $G\in \sepdholsp{q}(\ms)$. Then we say that $\tf$ satisfies the Functional Correlation Bound with rate $n^{-\gamma}$.
\end{defn}
A similar condition was introduced by Lepp\"{a}nen in \cite{leppanen2017functional} and further studied by Lepp\"{a}nen and Stenlund in \cite{leppanen2017billiards,leppanen2020sunklodas}. In particular, \cite{leppanen2017functional}  showed that functional correlation decay implies a multi-dimensional CLT with bounds on the rate of decay.
We are now ready to state the main results which we prove in this paper. 

The rate of decay of correlations of a dynamical system modelled by a Young tower is determined by the tails of the
return time to the base of the tower. Indeed, let $ \tf $ be a mixing transformation modelled by a two-sided Young tower with tails of the form $O(n^{-\beta})$
for some $\beta>1$. In~\cite{melbourne_terhesiu_2014} by using ideas from \cite{chazottes2012optimal,gouezelprivate}, it was shown that there exists $ C>0 $ such that
\[ \bigg|\int_M v\ w\circ \tf^n d\mu-\int_\ms v d\mu\int_\ms wd\mu\bigg|\le Cn^{-(\beta-1)}\dhnorm{v}\dhnorm{w} \]
for all $n\ge 1, v,w \in \dholsp(\ms)$.
Our first main result is that the \fcb{} holds with the same rate:
\begin{theorem}\label{thm:fcb_2_sided_yt}
	Let $\beta>1$. Let $\tf$ be a mixing transformation modelled by a two-sided Young tower whose return time has tails of the form $O(n^{-\beta})$. Then $\tf$ satisfies the \fcb{} with rate $n^{-(\beta-1)}$.
\end{theorem}
Given $v,w\in \dholsp(\ms)$ mean zero define
\begin{equation*}
	S_v(n)=\sum_{0\le i<n}v\circ \tf^i,\quad \bbS_{v,w}(n)=\sum_{0\le i<j<n}v\circ\tf^i\ w\circ \tf^j.
\end{equation*}
Our second main result is that the Functional Correlation Bound implies moment estimates for $ S_v(n) $ and $ \bbS_{v,w}(n) $. Let $ \dhnorm{\cdot}=\norminf{\cdot}+\dsemi{\cdot} $ denote the dynamically \holder{} norm.
\begin{theorem}\label{thm:moment_bd}
	Let $\gamma>1$. Suppose that $\tf$ satisfies the \fcb{} with rate $n^{-\gamma}$. Then there exists a constant $ C>0 $ such that for all $n\ge 1$, for any mean zero $ v,w\in \dholsp(\ms) $,
	\begin{enumerate}[label=(\alph*)]
		\item\label{item:moment_bd} $\pnorm{S_v(n)}_{2\gamma}\le Cn^{1/2}\dhnorm{v}$.
		\item\label{item:iterated_moment_bd} $\pnorm{\bbS_{v,w}(n)}_\gamma\le Cn\dhnorm{v}\dhnorm{w}$.
	\end{enumerate}
\end{theorem}
\begin{remark}
		As mentioned above, by \cite[Theorem~2.10]{chevyrev2020deterministic} to obtain deterministic homogenisation results it suffices to prove the iterated WIP and iterated moment bounds. Let $ \tf $ be a mixing transformation modelled by a two-sided Young tower with tails of the form $ O(n^{-\beta})$ for some $\beta>2$. By \cite{melbourne2016note}, the Iterated WIP holds for all \holder{} observables. Together Theorem~\ref{thm:fcb_2_sided_yt} and Theorem~\ref{thm:moment_bd} imply that for all $ \eta\in(0,1] $ there exists $ C>0 $ such that
		\begin{enumerate}[label=(\alph*)]
			\item\label{item:yt_moment_bd} $\pnorm{S_v(n)}_{2(\beta-1)}\le Cn^{1/2}\norm{v}_\eta$.
			\item\label{item:yt_iterated_moment_bd} $\pnorm{\bbS_{v,w}(n)}_{\beta-1}\le Cn\norm{v}_\eta\norm{w}_\eta$.
		\end{enumerate}
	for all mean zero $ v,w\in \holsp(\ms)$, giving the required control of iterated moments.
	\end{remark}
	\section{Young towers}\label{section:yt}
	\subsection{Prerequisites}\label{section:yt_prereq}
	Young towers were first introduced by L.-S. Young in \cite{young1998statistical,young1999recurrence}, as a broad framework to prove decay of correlations for nonuniformly hyperbolic maps. Our presentation follows~\cite{bmtmaps}. In particular, this framework does not assume uniform contraction along stable manifolds and hence covers examples such as billiards.

\textbf{Gibbs-Markov maps:} Let \((\bar Y,\bar\mu_Y)\) be a probability space and let $\bar F\map{\bar Y}{\bar Y}$ be ergodic and measure-preserving. Let $\alpha$ be an at most countable, measurable partition of $\bar Y$. We assume that there exist constants $K>0,\ \theta\in (0,1)$ such that for all elements $a\in \alpha$:
\begin{itemize}
	\item (Full-branch condition) The map $\bar F|_{a}\map{a}{\bar Y}$ is a measurable bijection.
	\item For all distinct $y,y'\in \bar Y$ the separation time 
	\begin{equation*}
		s(y,y')=\inf\{n\ge 0: \bar F^n y,\, \bar F^n y' \text{ lie in distinct elements of }\alpha\}<\infty.
	\end{equation*}
	\item \label{item:GM_bdd_distortion} Define $\zeta\map{a}{\R^+}$ by $\zeta=d\bar\mu_Y/(d\, (F|_a^{-1})_* \bar\mu_Y)$. We have $ |\log \zeta(y)-\log \zeta(y')|\le K\theta^{s(y,y')} $ for all $ y,y'\in a $.
\end{itemize}
Then we call $\bar{F}\map{\bar Y}{\bar Y}$ a full-branch Gibbs-Markov map.

\textbf{Two-sided Gibbs-Markov maps}\quad Let $(Y,d)$ be a bounded metric space with Borel probability measure $\mu_Y$ and let $F\map{Y}{Y}$ be ergodic and measure-preserving. Let $\bar{F}\map{\bar{Y}}{\bar{Y}}$ be a full-branch Gibbs-Markov map with associated measure $\bar\mu_Y$.

We suppose that there exists a measure-preserving semi-conjugacy $\bar{\pi}\map{Y}{\bar{Y}}$, so $\bar{\pi}\circ F=\bar{F}\circ \bar{\pi}$ and $\bar{\pi}_{*}\mu_Y=\bar{\mu}_Y.$ The separation time $s(\cdot,\cdot)$ on $\bar Y$ lifts to a separation time on $Y$ given by $s(y,y')=s(\bar\pi y,\bar\pi y')$. Suppose that there exist constants $K>0,\theta\in(0,1)$ such that
\begin{equation}\label{eq:two_sided_GM_contraction}
	d(F^n y,F^n y')\le K(\theta^n+\theta^{s(y,y')-n}) \text{ for all }y,y'\in Y,n\ge 0.
\end{equation}
Then we call $F\map{Y}{Y}$ a \textit{two-sided Gibbs-Markov map.}

\textbf{One-sided Young towers:} Let $\bar\phi\map{\bar Y}{\Z^+}$ be integrable and constant on partition elements of $\alpha$. We define the one-sided Young tower $\bar\Delta=\bar Y^{\bar\phi}$ and tower map $\bar f\map{\bar \Delta}{\bar \Delta}$ by  
\begin{equation}\label{eq:one-sided_yt_def}
	\bar\Delta=\{(\bar y,\ell)\in \bar Y\times \Z: 0\le \ell<\bar \phi(y)\},\quad\bar{f}(\bar y,\ell)=\begin{cases}
		(\bar y,\ell+1),& \ell<\bar\phi(y)-1,\\
		(\bar{F}\bar{y},0),& \ell=\bar\phi(y)-1.
	\end{cases}
\end{equation}
We extend the separation time $s(\cdot,\cdot)$ to $\bar\Delta$ by defining
\[s((\bar y,\ell),(\bar y',\ell'))=\begin{cases}
	s(\bar{y},\bar{y}'),& \ell=\ell',\\
	0,& \ell\ne \ell'.
\end{cases}\]
Note that for $\theta\in (0,1)$ we can define a metric by $d_\theta(\bar{p},\bar{q})=\theta^{s(\bar{p},\bar{q})}$.

Now, $\bar{\mu}_\Delta=(\bar{\mu}_Y \times \text{counting})/\int_{\bar{Y}}\bar{\phi} d\bar{\mu}_Y$ is an ergodic $\bar{f}$-invariant probability measure on $\bar{\Delta}$.
\vspace{1em}

\textbf{Two-sided Young towers}\quad Let $F\map{Y}{Y}$ be a two-sided Gibbs-Markov map and let $\phi\map{Y}{\Z^+}$ be an integrable function that is constant on $\bar{\pi}^{-1}a$ for each $a\in \alpha$. In particular, $\phi$ projects to a function $\bar{\phi}\map{\bar{Y}}{M}$ that is constant on partition elements of $\alpha$.

Define the one-sided Young tower $\bar{\Delta}=\bar{Y}^{\bar{\phi}}$ as in \eqref{eq:one-sided_yt_def}. Using $\phi$ in place of $\bar{\phi}$ and $F\map{Y}{Y}$ in place of $\bar{F}\map{\bar{Y}}{\bar{Y}}$, we define the \textit{two-sided Young tower} $\Delta=Y^{\phi}$ and tower map $f\map{\Delta}{\Delta}$ in the same way. Likewise, we define an ergodic $ f $-invariant probability measure on $\Delta$ by $\mu_\Delta=(\mu_Y \times \text{counting})/\int_{Y}\phi\, d\mu_Y$. 

We extend $\bar{\pi}\map{Y}{\bar{Y}}$ to a map $\bar{\pi}\map{\Delta}{\bar{\Delta}}$ by setting $\bar{\pi}(y,\ell)=(\bar{\pi}y,\ell)$ for all $(y,\ell)\in \Delta$. Note that $\bar{\pi}$ is a measure-preserving semi-conjugacy; $\bar{\pi}\circ f=\bar{f}\circ \bar{\pi}$ and $\bar{\pi}_{*}\mu_\Delta=\bar{\mu}_\Delta$. The separation time $s$ on $\bar{\Delta}$ lifts to $\Delta$ by defining $s(y,y)=s(\bar{\pi}y,\bar{\pi}y').$

We are now finally ready to say what it means for a map to be modelled by a Young tower:

Let $ \tf\map{\ms}{\ms} $ be a measure-preserving transformation on a probability space $ (\ms,\mu) $. Suppose that there exists $Y\subset M$ measurable with $\mu(Y)>0$ such that:
\begin{itemize}
	\item $F=T^{\phi}\map{Y}{Y}$ is a two-sided Gibbs-Markov map with respect to some probability measure $\mu_Y$.
	\item $\phi$ is constant on partition elements of $\bar{\pi}^{-1}\alpha$, so we can define Young towers $\Delta=Y^\phi$ and $\bar\Delta=\bar{Y}^{\bar{\phi}}$.
	\item The map $\pi_M\map{\Delta}{M}$, $\pi_M(y,\ell)=T^\ell y$ is a measure-preserving semiconjugacy.
\end{itemize}
Then we say that $\tf\map{\ms}{\ms}$ is \textit{modelled by a (two-sided) Young tower.}

From now on we fix $ \beta>1 $ and suppose that $ \tf\map{\ms}{\ms} $ is a mixing transformation modelled by a Young tower $ \Delta $ with tails of the form $ \mu_Y(\phi\ge n)=O(n^{-\beta}). $
\begin{remark}
	Here we have not assumed that the tower map $ f\map{\Delta}{\Delta} $ is mixing. However, as in \cite[Theorem~2.1, Proposition~10.1]{chernov1999decay} and \cite{bmtmaps} the a priori knowledge that $ \mu $ is mixing ensures that this is irrelevant.
\end{remark}
Let $ \psi_n(x)=\#\{j=1,\dots,n\colon f^j x\in \Delta_0\} $ denote the number of returns to $ \Delta_0=\{(y,\ell)\in\Delta:\ell=0\} $ by time $ n $. The following bound is standard, see for example \cite[Lemma~5.5]{korepanov2019explicit}.
\begin{lemma}\label{lemma:neg_exp_moments}
	Let $\theta\in (0,1)$. Then there exists a constant $D_1>0$ such that
	\[\pushQED{\qed}
	\int_{\Delta}\theta^{\psi_n}d\mu_\Delta\le D_1 n^{-(\beta-1)} \text{ for }n\ge 1.\qedhere
	\popQED
	\]
\end{lemma}
The transfer operator $L$ corresponding to $\bar{f}\map{\bar\Delta}{\bar\Delta}$ and $\bar{\mu}_\Delta$ is given pointwise by
\[(Lv)(x)=\sum_{\bar fz=x}g(z)v(z), \text{ where } g(y,\ell)=\begin{cases}
	\zeta(y),& \ell=\phi(y)-1,\\
	1,& \ell<\phi(y)-1
\end{cases}.\]
It follows that for $n\ge 1$, the operator $L^n$ is of the form $(L^n v)(x)=\sum_{\bar{f}^n z=x}g_n(z)v(z)$, where $g_n=\prod_{i=0}^{n-1}g\circ \bar f^i.$

We say that $z,z'\in \bar\Delta$ are in the same cylinder set of length $n$ if $\bar f^k z$ and $\bar f^k z'$ lie in the same partition element of $\bar\Delta$ for $0\le k\le n-1$. We use the following distortion bound (see e.g. \cite[Proposition~5.2]{korepanov2019explicit}):
\begin{prop}\label{prop:g_n-bdd-distortion}
	There exists a constant $K_1>0$ such that for all $n\ge1$, for all points $z,z'\in\bar\Delta$ which belong to the same cylinder set of length $n$,
	\[\pushQED{\qed}
	|g_n(z)-g_n(z')|\le Cg_n(z)d_\theta(\bar{f}^n z,\bar f^n z').\qedhere
	\popQED\]
\end{prop}
Let $ \theta\in(0,1). $ We say that $ v\map{\bar\Delta}{\R} $ is $ d_\theta $-Lipschitz if $ \norm{v}_\theta=\norminf{v}+\sup_{x\ne y}|v(x)-v(y)|/d_\theta(x,y)<\infty $. If $ f\map{\Delta}{\Delta} $ is mixing then by~\cite{young1999recurrence}, \[ \pnorm{L^n v-\int v\, d\bar\mu_\Delta}_1=O(n^{-(\beta-1)}\norm{v}_\theta) . \]
The same bound holds pointwise on $ \bar{\Delta}_0 $:
\begin{lemma}\label{lemma:transfer_op_ptwise_bd}
	Suppose that $ f\map{\Delta}{\Delta}$ is mixing. Then there exists $ D_2>0 $ such that for all $ d_\theta $-Lipschitz $ v\map{\bar\Delta}{\R} $, for any $ n\ge 1, $
	\[ \norminf{\indicd{\bar\Delta_0}L^n v-\int_{\bar\Delta} v\, d\bar\mu_\Delta}\le D_2 n^{-(\beta-1)}\norm{v}_\theta. \]
\end{lemma}
This is a straightforward application of operator renewal theory developed by Sarig \cite{sarig2002subexponential} and Gou{\"e}zel \cite{gouezel2004sharp,gouezel2004vitesse}. However, we could not find a reference to this result in the literature so we provide a proof.
\begin{proof}
	Define partial transfer operators $ T_n $ and $ B_n $ as in \cite[Section 4]{gouezel2005berry}. Then 
	\[ \indicd{\bar\Delta_0}L^n v=\sum_{k+b=n}T_k B_b v. \]
Define an operator $ \Pi $ by $ \Pi v=\int_{\bar\Delta_0}v\, d\bar\mu_\Delta $. Then as in the proof of \cite[Theorem 4.6]{gouezel2005berry} we can write $ T_k=\Pi +E_k $ where $ \norm{E_k}=O(k^{-(\beta-1)}). $ Moreover, by~\cite[Theorem 4.6]{gouezel2005berry}, $ \norm{B_b}=O(b^{-\beta}) $ and \[\sum_{b=0}^\infty \int_{\bar\Delta_0}B_b v\, d\bar\mu_\Delta =\int_{\bar\Delta}v\ d\bar\mu_\Delta. \]
	It follows that
	\begin{align*}
		\indicd{\bar\Delta_0}L^n v&=\sum_{k+b=n}\Pi B_b v+\sum_{k+b=n}E_k B_b v\\
		&=\sum_{b=0}^n \int_{\bar\Delta_0}B_b v\, d\bar\mu_\Delta+\sum_{k+b=n}E_k B_b v\\
		&=\int_{\bar\Delta}v\, d\bar\mu_\Delta -\sum_{b=n+1}^\infty \int_{\bar\Delta_0}B_b v\, d\bar\mu_\Delta+\sum_{k+b=n}E_k B_b v.
	\end{align*}
The conclusion of the lemma follows by noting that the expressions $ \sum^\infty_{b= n+1}b^{-\beta}$ and 
\[\sum_{k+b=n}(k+1)^{-(\beta-1)}(b+1)^{-\beta}\]
are both $ O(n^{-(\beta-1)})$.
\end{proof}
Finally we recall the class of observables on $ \ms $ that are of interest to us:

\textbf{Dynamically \holder{} observables}\quad Fix $ \theta\in(0,1) $. For $ v\map{\ms}{\R}$, define
\[ \dhnorm{v}=\norminf{v}+\dsemi{v}, \quad \dsemi{v}=\sup_{y,y'\in Y,y\ne y'}\sup_{0\le\ell<\phi(y)}\frac{|v(\tf^{\ell}y)-v(\tf^{\ell}y')|}{d(y,y')+\theta^{s(y,y')}}. \]
We say that $ v $ is dynamically \holder{} if $ \dhnorm{v}<\infty $ and denote by $ \dholsp(\ms) $ the space of all such observables.

It is standard (see e.g. \cite[Proposition~7.3]{bmtmaps}) that \holder{} observables are also dynamically \holder{} for the classes of dynamical systems that we are interested in:
\begin{prop}
	Let $ \eta\in(0,1] $ and let $ d_0 $ be a bounded metric on $ \ms $. Let $ \holsp(\ms) $ be the space of observables that are $ \eta $-\holder{} with respect to $ d_0 $. Suppose that there exists $ K>0 $, $ \gamma_0\in(0,1) $ such that $ d_0(\tf^\ell y,\tf^\ell y')\le K(d_0(y,y')+\gamma_0^{s(y,y')}) $ for all $ y,y'\in Y,0\le \ell<\phi(y). $ 
	
	Then $ \holsp(\ms) $ is continuously embedded in $ \dholsp(\ms) $ where we may choose any $ \theta\in[\gamma_0^\eta,1) $ and $ d=d_0^{\eta'} $ for any $ \eta'\in(0,\eta]. $\qedhere
\end{prop}
\subsection{Reduction to the case of a mixing Young tower}
In proofs involving Young towers it is often useful to assume that the Young tower is mixing, i.e.\ $gcd \{\phi(y):y\in Y\}=1. $ Hence in subsequent subsections we focus on proving the \fcb{} under this assumption:
\begin{lemma}\label{lemma:fcb_mixing_case}
	Suppose that $ \tf$ is modelled by a mixing two-sided Young tower whose return time has tails of the form $ O(n^{-\beta}) $. Then $ \tf $ satisfies the \fcb{} with rate $ n^{-(\beta-1)} $.
\end{lemma}
\begin{proof}[Proof of Theorem~\ref{thm:fcb_2_sided_yt}]
	Let $ d=\gcd\{\phi(y):y\in Y\} $. Set $ \tf'=\tf^d $ and $ \phi'=\phi/d. $ Construct a mixing two-sided Young tower $ \Delta'=Y^{\phi'} $, with tower measure $ \mu'_\Delta. $ Define $ \pi'_{\ms}\map{\Delta'}{\ms} $ by $ \pi'_{\ms}(y,\ell)=(\tf')^\ell y.$ Then $ \tf' $ is modelled by $ \Delta' $ with ergodic, $ \tf' $-invariant measure $ (\pi'_{\ms})_{*}\mu'_\Delta.$ 
	Now by assumption the measure $ \mu $ is mixing so by the same argument as in \cite[Section~4.1]{bmtmaps} we must have $ \mu=(\pi'_\ms)_* \mu'_{\Delta} $.
	
	Let $ G\in \sepdholsp{q}(\ms) $ and fix integers $ 0\le n_0\le \cdots\le n_{q-1} $. Define $ n'_i=[n_i/d], r_i=n_i\! \mod d $. We need to bound 
	\begin{align*}
		\nabla G&=\int_\ms  G(\tf^{n_0}x,\dots,\tf^{n_{q-1}}x)d\mu(x)\\
		&\quad-\int_{\ms^2} G(\tf^{n_0}x_0,\dots,\tf^{n_{p-1}}x_0,\tf^{n_p}x_1,\dots,\tf^{n_{q-1}} x_1)d\mu(x_0)d\mu(x_1).
	\end{align*}
Define $ G'\map{\ms^q}{\R} $ by $ G'(x_0,\dots,x_{q-1})=G(\tf^{r_0}x_0,\dots,\tf^{r_{q-1}}x_{q-1}). $ Then 
	\begin{align*}
	\nabla G&=\int_\ms  G'((\tf')^{n'_0}x,\dots,(\tf')^{n'_{q-1}}x)d\mu(x)\\
	&\quad-\int_{\ms^2} G'({(\tf')}^{n'_0} x_0,\dots,(\tf')^{n'_{p-1}}x_0,(\tf')^{n'_p}x_1,\dots,(\tf')^{n'_{q-1}} x_1)d\mu(x_0)d\mu(x_1).
\end{align*}
Let $ [\cdot]_{\mathcal{H}'} $ denote the dynamically \holder{} seminorm as defined with $ \tf',\phi' $ in place of $ \tf,\phi $. Then by Lemma~\ref{lemma:fcb_mixing_case},
\begin{align*}
	|\nabla G|&\le C(n'_p-n'_{p-1})^{-\gamma}\biggl(\norminf{G'}+\sum_{i=0}^{q-1} [G']_{\mathcal{H}',i}\biggr)\\
	&\le Cd^{\gamma}(n_p-n_{p-1}-d)^{-\gamma}\biggl(\norminf{G}+\sum_{i=0}^{q-1}[G']_{\mathcal{H}',i}\biggr)
\end{align*}
Now fix $ 0\le i<q.$ Let $ x_0,\dots,x_{q-1}\in \ms $ and write
\begin{align*}
	v'(y)&=G'(x_0,\dots,x_{i-1},y,x_{i+1},\dots,x_{q-1})\\
	&=G(\tf^{r_0}x_0,\dots,\tf^{r_{i-1}}x_{i-1},\tf^{r_i}y,\tf^{r_{i+1}}x_{i+1},\dots,\tf^{r_{q-1}}x_{q-1})=v(\tf^{r_i}y).
\end{align*}
Let $ y,y'\in Y $ and $ 0\le \phi'(y)<\ell. $ Then
\[ |v'((\tf')^{\ell}y)-v'((\tf')^{\ell} y')|=|v(\tf^{d\ell +r_i}y)-v(\tf^{d\ell +r_i}y')|\le \sdsemi{G}{i}(d(y,y')+\theta^{s(y,y')}), \]
so $[G']_{\mathcal{H}',i} \le \sdsemi{G}{i}. $
\end{proof}
\subsection{Approximation by one-sided functions}\label{section:one_sided_yt_approx}
Let $ 0\le p<q $ and $ 0\le n_0\le \cdots\le n_{q-1} $ be integers and consider a function $ G\in \sepdholsp{q}(\ms) $. We wish to bound
\begin{align*}
	\nabla G&=\int_\ms  G(\tf^{n_0}x,\dots,\tf^{n_{q-1}}x)d\mu(x)\\
	&\qquad -\int_{\ms^2} G(\tf^{n_0}x_0,\dots,\tf^{n_{p-1}}x_0,\tf^{n_{p}}x_1,\dots,\tf^{n_{q-1}} x_1)d\mu^2(x_0,x_1).
\end{align*}
Now since $\pi_\ms\map{\Delta}{\ms}$ is a measure-preserving semiconjugacy
\begin{equation}\label{eq:nabla_tilde_H_def}
	\nabla G=\int_\Delta \widetilde H(x,f^{n_p}x)d\mu_\Delta(x)-\int_{\Delta^2} \widetilde H(x_0,x_1)d\mu^2_{\Delta}(x_0,x_1)=\nabla \widetilde{H}
\end{equation}
where $ \widetilde H \map{\Delta^2}{\R}$ is given by 
\[ \widetilde H(x,y)=\widetilde G(f^{n_0}x,f^{n_1}x,\dots,f^{n_{p-1}}x,f^{k_{p}}y,f^{k_{p+2}}y,\dots,f^{k_{q-1}}y), \]
where $ \stilde G=G\circ \pi_\ms $ and $ k_i=n_i-n_p $.

Let $ R\ge 1 $. We approximate $\widetilde{H}(f^R \cdot,f^R \cdot)$ by a function $\widetilde{H}_R$ that projects down onto $\bar{\Delta}$. Our approach is based on ideas from Appendix B of \cite{melbourne_terhesiu_2014}.

Recall that $ \psi_R(x)=\#\{j=1,\dots,R: f^j x\in \Delta_0\} $ denotes the number of returns to $\Delta_0=\{(y,\ell)\in \Delta: \ell=0\}$ by time $R$. Let $\mathcal{Q}_R$ denote the at most countable, measurable partition of $\Delta$ with elements of the form $\{x'\in \Delta: s(x,x')>2\psi_R(x)\}$, $x\in \Delta$. Choose a reference point in each partition element of $\mathcal{Q}_R$. For $ x\in \Delta $ let $ \hat{x} $ denote the reference point of the element that $ x $ belongs to. Define $ \widetilde{H}_R\map{\Delta^2}{\R} $ by
\begin{equation*}
	\widetilde{H}_R(x,y)=\widetilde{G}(f^R\widehat{f^{n_{0}}x},\dots,f^R\widehat{f^{n_{p-1}}x},f^R\widehat{f^{k_p}y},\dots,f^R\widehat{f^{k_{q-1}}y}).
\end{equation*}

\begin{prop}\label{prop:H_R_approx}
	The function $ \widetilde{H}_R $ lies in $ L^\infty(\Delta^2) $ and projects down to a function $ \bar{H}_R\in L^\infty(\bar\Delta^2) $. Moreover, there exists a constant $K_2>0$ depending only on $ \tf\map{\ms}{\ms} $ such that,
\begin{enumerate}
	\item\label{item:H_R_norminf} 
	$ \norminf{\bar{H}_R}=\norminf{\stilde{H}_R}\le \norminf{G}. $
	\item\label{item:H_R-H_diff}
	 For all $ x,y\in \Delta $,
	\[ |\stilde{H}(f^R x,f^R y)-\stilde{H}_R(x,y)|\le K_2\biggl(\sum_{i=0}^{p-1} \sdsemi{G}{i}\,\theta^{\psi_R(f^{n_i}x)}+\sum_{i=p}^{q-1} \sdsemi{G}{i}\,\theta^{\psi_R(f^{k_i}y)}\biggr). \]
	\item \label{item:H_R-transfer_op}
	For all $ \bar y\in \bar\Delta, $
	\[ \norm{L^{R+n_{p-1}}\bar{H}_R(\cdot,\bar y)}_\theta\le K_2\biggl(\norminf{G}+\sum_{i=0}^{p-1} \sdsemi{G}{i}\biggr). \]
\end{enumerate}
\end{prop}
Here we recall that $ \norm{\cdot}_\theta $ denotes the $ d_\theta $-Lipschitz norm, which is given by $ \norm{v}_\theta=\norminf{v}+\sup_{x\ne y}|v(x)-v(y)|/d_\theta(x,y)$ for $ v\map{\bar\Delta}{\R} $.
\begin{proof}
	We follow the proof of Proposition~7.9 in~\cite{bmtmaps}.
	
	By definition $ \widetilde{H}_R $ is piecewise constant on a measurable partition of $ \Delta^2 $. Moreover, this partition projects down to a measurable partition on $ \bar\Delta $, since it is defined in terms of $ s $ and $ \psi_R $ which both project down to $ \bar\Delta $. It follows that $ \bar{H}_R $ is well-defined and measurable. Part \ref*{item:H_R_norminf} is immediate.
	
	Let $ x,y\in \Delta $. Write $\stilde{H}(f^R x,f^R y)-\stilde{H}_R(x,y) =I_1+I_2 $ where
	\begin{align*}
		I_1&=\widetilde{G}(f^R f^{n_{0}}x,\dots,f^R f^{n_{p-1}}x,f^R f^{k_{p}}y,\dots,f^R f^{k_{q-1}}y)\\ 
		&\qquad\qquad- \widetilde{G}(f^R\widehat{f^{n_{0}}x},\dots,f^R\widehat{f^{n_{p-1}}x},f^R f^{k_{p}}y,\dots,f^R f^{k_{q-1}}y),\\
		I_2&=\widetilde{G}(f^R\widehat{f^{n_{0}}x},\dots,f^R\widehat{f^{n_{p-1}}x},f^R f^{k_{p}}y,\dots,f^R f^{k_{q-1}}y)\\
		&\qquad\qquad - \widetilde{G}(f^R\widehat{f^{n_{0}}x},\dots,f^R\widehat{f^{n_{p-1}}x},f^R\widehat{f^{k_{p}}y},\dots,f^R\widehat{f^{k_{q-1}}y}).
	\end{align*}
Let $ a_i=f^{n_i}x$ and $ b_i=f^R f^{k_i}y$. By successively substituting $ a_i $ by $ \hat{a}_i $,
\begin{align}
	I_1&=\stilde{G}(f^R a_0,\dots,f^R a_{p-1},b_{p},\dots,b_{q-1})-\stilde{G}(f^R \hat{a}_0,\dots,f^R \hat{a}_{p-1},b_p,\dots,b_{q-1})\nonumber\\
	&=\sum_{i=0}^{p-1} \bigl(\widetilde G(f^R a_0,\dots,f^R a_{i-1}, f^R a_i, f^R \hat{a}_{i+1},f^R \hat{a}_{p-1},b_p,\dots,b_{q-1})\nonumber\\
	&\qquad\quad - \widetilde G(f^R a_0,\dots,f^R a_{i-1}, f^R \hat{a}_i, f^R \hat{a}_{i+1},f^R \hat{a}_{p-1},b_p,\dots,b_{q-1}) \bigr)\nonumber\\
	&=\sum_{i=0}^{p-1}\bigl(\tilde v_i(f^R a_i)-\tilde v_i(f^R \hat a_i)\bigr)
\end{align}
where $\tilde v_i(x)=\tilde G(f^R a_0,\dots,f^R a_{i-1},x,f^R \hat a_{i+1},\dots,f^R \hat a_{p-1},b_p,\dots,b_{q-1}). $

Fix $ 0\le i<p $. Since $ a_i $ and $ \hat{a}_i $ are in the same partition element, $ s(a_i,\hat{a}_i)>2\psi_R(a_i) $. Write $ a_i=(y,\ell), \hat a_i=(\hat y,\ell). $ Then $ f^R a_i=(F^{\psi_R(a_i)}y,\ell_1)$ and similarly $f^R \hat a_i=(F^{\psi_R(a_i)}\hat y,\ell_1)$, where $ \ell_1=\ell+R-\Phi_{\psi_R(a_i)}(y) $. (Here, $ \Phi_k=\sum_{j=0}^{k-1}\phi\circ F^k $.) Now by the definition of $ \sdsemi{G}{i} $ and \eqref{eq:two_sided_GM_contraction},
\begin{align*}
	|\tilde v_i(f^R a_i)-\tilde v_i(f^R \hat a_i)|&=|\tilde v_i(F^{\psi_R(a_i)}y,\ell_1)-\tilde v_i(F^{\psi_R(a_i)}\hat y,\ell_1)|\\
	&\le \sdsemi{G}{i}(d(F^{\psi_R(a_i)}y,F^{\psi_R(a_i)}y')+\theta^{s(F^{\psi_R(a_i)}y,F^{\psi_R(a_i)}y')})\\
	&\le (K+1)\sdsemi{G}{i}(\theta^{\psi_R(a_i)}+\theta^{s(a_i,a_i')-\psi_R(a_i)})\\
	&\le 2(K+1)\sdsemi{G}{i}\theta^{\psi_R(a_i)}.
\end{align*}

Thus
\[ |I_1|\le 2(K+1)\sum_{i=0}^{p-1} \sdsemi{G}{i}\theta^{\psi_R(f^{n_i}x)}. \]
By a similar argument,
\[|I_2|\le 2(K+1)\sum_{i=p}^{q-1} \sdsemi{G}{i}\theta^{\psi_R(f^{k_i}y)},\]
completing the proof of \ref*{item:H_R-H_diff}.

Let $\bar x,\bar x',\bar y\in \bar\Delta. $ Recall that 
\[ L^{R+n_{p-1}}\bar{H}_R(\cdot,\bar{y})(\bar x)=\sum_{\bar{f}^{R+n_{p-1}}\bar z=\bar x}g_{R+n_{p-1}}(\bar z)\bar{H}_R(\bar z,\bar y) .\]
It follows that $ \norminf{L^{R+n_{p-1}}\bar{H}_R(\cdot,\bar{y})}\le \norminf{\bar{H}_R}\le \norminf{G}.$ If $ d_\theta(\bar x,\bar x')=1 $, then
\[|L^{R+n_{p-1}}\bar{H}_R(\cdot,\bar{y})(\bar x)-L^{R+n_{p-1}}\bar{H}_R(\cdot,\bar{y})(\bar x')|\le 2\norminf{G}=2\norminf{G}d_\theta(\bar x,\bar x').  \]
Otherwise, we can write $ L^{n_{p-1}+R}\bar{H}_R(\cdot,\bar{y})(\bar x)-L^{n_{p-1}+R}\bar{H}_R(\cdot,\bar{y})(\bar x')=J_1+J_2 $ where
\begin{align*}
	J_1&=\sum_{\bar{f}^{n_{p-1}+R}\bar z=\bar x}\bigl(g_{n_{p-1}+R}(\bar z)-g_{n_{p-1}+R}(\bar z')\bigr)\bar{H}_R(\bar z,\bar y),\\
	J_2&=\sum_{\bar{f}^{n_{p-1}+R}\bar z'=\bar x'}  g_{n_{p-1}+R}(\bar z')\bigl(\bar{H}_R(\bar z,\bar y)-\bar{H}_R(\bar z',\bar y)\bigr).
\end{align*}
Here, as usual we have paired preimages $ \bar z,\bar z' $ that lie in the same cylinder set of length $ n_{p-1}+R $. By bounded distortion (Proposition~\ref{prop:g_n-bdd-distortion}), $ |J_1|\le C\norminf{G}d_\theta(\bar x,\bar x'). $ We claim that $ |\bar{H}_R(\bar z,\bar y)-\bar{H}_R(\bar z',\bar y)|\le K_2\sum_{i=0}^{p-1} \sdsemi{G}{i}d_\theta(\bar x,\bar x') $. It follows that $ |J_2|\le K_2\sum_{i=0}^{p-1} \sdsemi{G}{i}d_\theta(\bar x,\bar x'). $

It remains to prove the claim.
Choose points $ z,z',y\in \Delta $ that project to $\bar z,\bar z',\bar y $. Let $ a_i=f^{n_i}z,a_i'=f^{n_i}z',b_i=f^{R+n_i}y. $ As in part~\ref*{item:H_R-H_diff},
\begin{align*}
	\bar{H}_R(\bar z,\bar y)-\bar{H}_R(\bar z',\bar y)=\stilde{H}_R(z,y)-\stilde{H}_R(z',y)
	=\sum_{i=0}^{p-1}(\tilde w_i(f^R \hat a_i)-\tilde w_i(f^R \hat a'_i))
\end{align*}
where $ \tilde w_i(x)=\tilde G(f^R \hat a_0,\dots,\hat a_{i-1}, x,f^R \hat a'_{i+1},\dots,\hat a'_{p-1},\hat b_p,\dots,\hat b_{q-1}) $. 

Let $ 0\le i< p $. We bound $ E_i=\tilde w_i(f^R \hat a_i)-\tilde w_i(f^R \hat a'_i) $. Without loss suppose that 
\[ \psi_R(\hat{a}'_i)\ge s(\hat a_i,\hat a'_i)-\psi_R(\hat a_i), \]
for otherwise $\hat a_i $ and $\hat a'_i $ are reference points of the same partition element so $ \hat a_i=\hat a'_i $ and $ E_i=0 $. Now as in part~\ref*{item:H_R-H_diff}, \[ E_i\le (K+1) (\theta^{\psi_R(\hat a_i)}+\theta^{s(\hat a_i,\hat a'_i)-\psi_R(\hat a_i)}).\]
Note that
\begin{align*}
	s(\hat a_i,\hat a'_i)-\psi_R(\hat a_i)&\ge \min\{s(\hat a_i, a_i),s(a_i, a'_i),s(a'_i,\hat a'_i)\}-\psi_R(\hat a_i).
\end{align*}
Since $ \bar z,\bar z' $ lie in the same cylinder set of length $ R+n_{p-1}$, we have $ \psi_R(a_i)=\psi_R(a'_i) $ and  
\begin{align*}
	s(a_i,a'_i)=s(\bar f^{n_i}\bar z,\bar f^{n_i}\bar z')&=s(\bar x,\bar x')+\psi_{R+n_{p-1}-n_i}(\bar{f}^{n_i}\bar z)\\
	&\ge s(\bar x,\bar x')+\psi_R(a_i).
\end{align*}
Now $ a_i $ and $ \hat{a}_i $ are contained in the same partition element so $ s(\hat a_i,a_i)-\psi_R(\hat a_i) \ge \psi_R(\hat a_i)$ and 
\[ \psi_R(\hat a_i)=\psi_R(a_i)=\psi_R(a_i')=\psi_R(\hat a'_i). \]

Hence $ s(\hat a_i,\hat a_i')-\psi_R(\hat a_i)\ge \min\{s(\bar x,\bar x'),\psi_R(a_i)\}$. It follows that $ E_i\le 2(K+1)\theta^{s(\bar x,\bar x')} $, completing the proof of the claim.
\end{proof}
\subsection{Proof of Lemma~\ref{lemma:fcb_mixing_case}}
We continue to assume that $ \beta>1 $ and that $ \mu_Y(\phi\ge n)=O(n^{-\beta}).$ We also assume that $ \gcd\{\phi(y):y\in Y\}=1 $ so that $ f\map{\Delta}{\Delta} $ is mixing.
\begin{lemma}\label{lemma:decay_correlations_generalisation}
	Let $ \theta\in(0,1). $ There exists $ D_3>0 $ such that for any $ V\in L^{\infty}(\bar\Delta^2) $,
	\begin{equation*}
		\biggl|\int_{\bar\Delta} V(x,\bar f^n x)d\bar\mu_\Delta(x)
		-\int_{\bar\Delta^2} V(x_0,x_1)d\bar\mu^2_\Delta(x_0,x_1)\biggr|\le D_3 n^{-(\beta-1)}\sup_{y\in \bar\Delta}\norm{V(\cdot,y)}_\theta
	\end{equation*}
	for all $ n\ge 1. $
\end{lemma}
\begin{remark}
	Let $ V(x,y)=v(x)w(y) $ where $v$ is $ d_\theta $-Lipschitz and $ w\in L^\infty(\bar\Delta) $. Then we obtain that
	\[ \left|\int_{\bar\Delta}v\, w\circ \bar{f}^n d\bar\mu_{\Delta}-\int_{\bar\Delta}v\, d\bar\mu_\Delta\int_{\bar\Delta}w\, d\bar\mu_\Delta\right|\le D_3 n^{-(\beta-1)}\norm{v}_\theta\norminf{w}, \]
	so Lemma~\ref{lemma:decay_correlations_generalisation} can be seen as a generalisation of the usual upper bound on decay of correlations for observables on the one-sided tower $ \bar\Delta $.
\end{remark}
\begin{remark}
	Our proof of Lemma~\ref{lemma:decay_correlations_generalisation} is based on ideas from \cite[Section~4]{chazottes2012optimal}. However, we have chosen to present the proof in full because (i) our assumptions are weaker, in particular we only require $ \beta>1 $ instead of $ \beta>2 $ and $V$ need not be separately $d_\theta$-Lipschitz and (ii) we avoid introducing Markov chains.
\end{remark}
\begin{proof}[Proof of Lemma~\ref{lemma:decay_correlations_generalisation}]
	Write $ v(x)=V(x,\bar f^n x) $ so
	\begin{align*}
		\int_{\bar\Delta} V(x,f^n x)\, d\bar\mu_\Delta(x)&=\int_{\bar\Delta} v \, d\bar\mu_\Delta=\int_{\bar\Delta} L^n v\, d\bar\mu_\Delta\\
		&=\int_{\bar\Delta} \sum_{\bar f^n  z=x}g_n(z)V(z,\bar{f}^n z) d\bar\mu(x)\\
		&=\int_{\bar\Delta} \sum_{\bar f^n z=x}g_n(z)V(z,x) d\bar\mu_\Delta(x)
		=\int_{\bar\Delta} (L^n u_x)(x)
		\, d\bar\mu_\Delta(x).
	\end{align*}
where $ u_x(z)=V(z,x)$.
Let $\bar \Delta_\ell=\{(y,j)\in \bar \Delta: j=\ell\} $ denote the $ \ell $-th level of $ \bar\Delta $. It follows that we can decompose
\[ \int_{\bar\Delta} V(x,\bar f^n x)d\bar\mu_\Delta(x)
-\int_{\bar\Delta^2} V(x_0,x_1)d\bar\mu^2_\Delta(x_0,x_1)=\sum_{\ell\ge 0}A_\ell \]
where
\[A_\ell=\int_{\bar\Delta_\ell} \bigg((L^n u_x)(x)-\int_{\bar \Delta}V(z,x)d\bar\mu_\Delta(z)\bigg)d\bar\mu_\Delta(x).\]
For all $ \ell\ge 0 $, 
\[ |A_\ell|\le 2\norminf{V}\bar\mu_\Delta(\bar \Delta_\ell)=2\norminf{V}\frac{\bar\mu_Y(\phi>\ell)}{\int \phi d\bar\mu_Y}=O(\norminf{V} (\ell+1)^{-\beta}). \]
Hence,
\[ \sum_{\ell\ge n/2}|A_\ell|=O\big(\! \norminf{V}n^{-(\beta-1)}\big). \]
Let $ x\in \bar\Delta_\ell $, $ \ell\le n $. Then $ (L^{n} u_x)(x)=(L^{n-\ell}u_x)(x_0) $ where $ x_0\in \bar\Delta_0 $ is the unique preimage of $ x $ under $ \bar f^\ell $. Thus by Lemma~\ref{lemma:transfer_op_ptwise_bd},
\[ |A_\ell|\le \int_{\bar\Delta_\ell}D_2 (n-\ell)^{-(\beta-1)}\norm{V(\cdot,x)}_\theta d\bar\mu_{\Delta}\le D_2(n-\ell)^{-(\beta-1)}\sup_{y\in \bar\Delta}\norm{V(\cdot,y)}_\theta\bar\mu_{\Delta}(\bar \Delta_\ell). \]
Hence,
\[ \sum_{\ell\le n/2}|A_\ell|\le D_2(n/2)^{-(\beta-1)}\sup_{y\in \bar\Delta}\norm{V(\cdot,y)}_\theta, \]
completing the proof.
\end{proof}
\begin{proof}[Proof of Lemma~\ref{lemma:fcb_mixing_case}]
	Recall that we wish to bound
	\[ \nabla \widetilde{H}=\int_\Delta \widetilde H(x,f^{n_{p}}x)d\mu_\Delta(x)-\int_{\Delta^2} \widetilde H(x_0,x_1)d\mu^2_{\Delta}(x_0,x_1). \]
	
	Without loss take $ n_p-n_{p-1}\ge 2 $. Let $ R=[(n_p-n_{p-1})/2] $. Write $ \nabla \widetilde{H}=I_1+I_2+\nabla \bar H_R $ where
	\begin{align*}
		I_1&=\int_\Delta \widetilde H(x,f^{n_p} x)d\mu_\Delta(x)-\int_\Delta \widetilde H_R(x,f^{n_p} x)d\mu_\Delta(x),\\
		I_2&=\int_\Delta \widetilde{H}_R(x_0,x_1)d\mu_\Delta^2(x_0,x_1)-\int_{\Delta^2} \widetilde H(x_0,x_1)d\mu^2_{\Delta}(x_0,x_1),\\
		\nabla \bar{H}_R&=\int_{\Delta} \widetilde H_R(x,f^{n_p}x)d\mu_\Delta(x)-\int_{\Delta^2} \widetilde{H}_R(x_0,x_1)d\mu_\Delta^2(x_0,x_1) \\
		&=\int_{\bar\Delta} \bar H_R(x,\bar f^{n_p}x)d\bar\mu_{\Delta}(x)-\int_{\bar\Delta^2} \bar{H}_R(x_0,x_1)d\bar\mu_\Delta^2(x_0,x_1) .
	\end{align*}
	Now by Proposition~\ref{prop:H_R_approx}\ref*{item:H_R-H_diff} and Lemma~\ref{lemma:neg_exp_moments},
	\begin{align}
		|I_1|&=\bigg|\int_\Delta \widetilde H(f^R x,f^{R+n_p}x)d\mu_\Delta(x)-\int_\Delta \widetilde H_R(x,f^{n_p}x)d\mu_\Delta(x)\bigg|
		\nonumber\\
		&\le K_2 \int_\Delta\biggl(\sum_{i=0}^{p-1} \sdsemi{G}{i}\theta^{\psi_R(f^{n_i}x)}+\sum_{i=p}^{q-1} \sdsemi{G}{i}\theta^{\psi_R(f^{n_{p}+k_i}x)}\biggr)d\mu_\Delta(x)\nonumber\\
		&= K_2 \sum_{i=0}^{q-1} \sdsemi{G}{i}\int_\Delta \theta^{\psi_R}d\mu_\Delta\le K_2 D_1\sum_{i=0}^{q-1} \sdsemi{G}{i}R^{-(\beta-1)}.\label{eq:generalised-decay-I_1-bd}
	\end{align}
Similarly,
\begin{equation}
	|I_2|\le K_2 D_1\sum_{i=0}^{q-1} \sdsemi{G}{i} R^{-(\beta-1)}.\label{eq:generalised-decay-I_2-bd}
\end{equation}
Now let $ u_y(z)=\bar{H}_R(z,y) $ and $ V(x,y)=(L^{n_{p-1}+R}u_y)(x) $. Then
\begin{equation}\label{eq:V-H_R-double-int-identity}
	\int_{\bar\Delta^2} V(x_0,x_1)\, d\bar\mu^2_\Delta(x_0,x_1)=\int_{\bar\Delta^2}\bar{H}_R(x_0,x_1)\, d\bar\mu_\Delta^2(x_0,x_1)
\end{equation}
and
\begin{align*}
 V(x,\bar{f}^{n_{p}-n_{p-1}-R}x) &= 
 \sum_{\bar f^{n_{p-1}+R} z=x}g_{n_{p-1}+R}(z)\bar{H}_R(z,\bar{f}^{n_{p}-n_{p-1}-R}x)\\
	&=\sum_{\bar f^{n_{p-1}+R} z=x}g_{n_{p-1}+R}(z)\bar{H}_R(z,\bar{f}^{n_{p}}z) =(L^{n_{p-1}+R}\hat{u})(x)
\end{align*}
where $ \hat u(z)=\bar{H}_R(z,\bar f^{n_{p}}z) $. Hence
\begin{align}
	\int_{\bar\Delta}V(x,\bar{f}^{n_{p}-n_{p-1}-R}x)d\bar \mu_\Delta(x)&=\int_{\bar \Delta} L^{n_{p-1}+R}\hat u \, d\bar\mu_\Delta \nonumber\\
	&=\int_{\bar \Delta} \hat u \, d\bar\mu_\Delta=\int_{\bar\Delta} \bar{H}_R(x,\bar f^{n_{p}}x)\, d\bar\mu_{\Delta}(x). \label{eq:V-H_R-int-identity}
\end{align}
Now by Proposition~\ref{prop:H_R_approx}\ref*{item:H_R-transfer_op}, $ \sup_{y\in \bar\Delta}\norm{V(\cdot,y)}_{\theta}\le K_2(\norminf{G}+\sum_{i=0}^{p-1} \sdsemi{G}{i}). $ By Lemma~\ref{lemma:decay_correlations_generalisation}, \eqref{eq:V-H_R-double-int-identity} and \eqref{eq:V-H_R-int-identity} it follows that
\begin{align}
	|\nabla \bar{H}_R|&=\bigg|\int_{\bar\Delta} V(x,\bar{f}^{n_{p}-n_{p-1}-R}x)d\bar\mu_\Delta(x)-\int_{\bar\Delta^2} V(x_0,x_1)d\bar\mu^2_\Delta(x_0,x_1)\bigg|\nonumber
	\\
	&\le K_2 D_3\bigg(\norminf{G}+\sum_{i=0}^{p-1} \sdsemi{G}{i}\bigg)(n_{p}-n_{p-1}-R)^{-(\beta-1)}.\label{eq:generalised-decay-nabla_H_R-bd}
\end{align}
Recall that $ R=[(n_{p}-n_{p-1})/2] $. Hence $ n_{p}-n_{p-1}-R\ge R$. By combining \eqref{eq:generalised-decay-I_1-bd}, \eqref{eq:generalised-decay-I_2-bd} and \eqref{eq:generalised-decay-nabla_H_R-bd} it follows that
\[|\nabla \widetilde{H}|\le K_2(2D_1+D_3) \sum_{i=0}^{q-1} \sdsemi{G}{i}([(n_{p}-n_{p-1})/2])^{-(\beta-1)}, \]
as required.
\end{proof}
\section{An abstract weak dependence condition}\label{section:abstract_weak_dep}
The \fcb{} can be seen as a weak dependence condition. Let $k\ge 1$ and consider $ k $ disjoint blocks of integers $\{\ell_i,\ell_i+1,\dots,u_i\}$, $0\le i< k$ with $ \ell_i\le u_i<\ell_{i+1}. $ Consider random variables $ X_i $ on $ (\ms,\mu) $ of the form
\begin{equation*}
	X_i(x)=\varPhi_i(\tf^{\ell_i}x,\dots,\tf^{u_i}x)
\end{equation*}
where $ \varPhi_i\in \sepdholsp{u_i-\ell_i+1}(\ms) $, $ 0\le i< k. $

When the gaps $\ell_{i+1}-u_i$ between blocks are large, the random variables $ X_0,\dots,X_{k-1} $ are weakly dependent. Let $\widehat{X}_0,\dots,\widehat{X}_{k-1}$ be independent random variables with $\widehat{X}_i{\eqd X_i}$.
\begin{lemma}\label{lemma:fcb_weak_dep}
	Suppose that $\tf$ satisfies the Functional Correlation Bound with rate $n^{-\gamma}$ for some $ \gamma>0 $. Let $R=\max_i \norminf{\varPhi_i}$. Then for all Lipschitz $ F\map{[-R,R]^k}{\R} $,
	\begin{multline*}
		\big|\Ex[\mu]{F(X_0,\dots,X_{k-1})}-\Ex{F(\shat{X}_0,\dots,\shat{X}_{k-1})}\big|\\
		\le C\sum_{r=0}^{k-2}(\ell_{r+1}-u_r)^{-\gamma}\biggl(\norminf{F}+\Lip(F)\sum_{i=0}^{k-1}\sum_{j=0}^{u_i-\ell_i}\sdsemi{\varPhi_i}{j}\biggr),
	\end{multline*}
	where $C>0$ only depends on $\tf\map{\ms}{\ms}$.
\end{lemma}
\begin{proof}
	We proceed by induction on $ k$. For $ k=1 $ the inequality is trivial. Assume that this lemma holds for $ k\ge 1 $.
	
	Consider an enriched probability space which contains independent copies of $ \{X_i\} $ and $ \{\shat{X}_i\} $. Write 
	\[
	\Ex[\mu]{F(X_0,\dots,X_{k})}-\Ex{F(\shat{X}_0,\dots,\shat{X}_{k})}=I_1+I_2
	\]
	where
	\begin{align*}
		I_1&=\Ex{F(X_0,\dots,X_{k-1},\shat{X}_{k})}-\Ex{F(\shat{X}_0,\dots,\shat{X}_{k})},\\
		I_2&=\Ex[\mu]{F(X_0,\dots,X_{k})}-\Ex{F(X_0,\dots,X_{k-1},\shat{X}_{k})}.
	\end{align*}
	Since $ \shat{X}_{k}\eqd X_{k} $ and $ \shat{X}_{k}$ is independent of $ X_0,\dots,X_{k-1} $ and $ \shat X_0,\dots,\shat X_{k-1} $,
	\begin{equation*}
		I_1=\int_M \big(\Ex[\mu]{F\big(X_0,\dots,X_{k-1},X_{k}(y)\big)}-\Ex{F\big(\shat X_0,\dots,\shat X_{k-1},X_{k}(y)\big)}\big) d\mu(y).
	\end{equation*}
	Let $ y\in \ms. $ The function $ F_y=F(\cdot,\dots,\cdot,X_{k}(y))\map{M^{k}}{\R} $ satisfies $ \Lip(F_y)\le \Lip(F) $. Hence by the inductive hypothesis,
	\begin{equation*}
		\begin{split}
			|I_1|&\le \int\big|\Ex[\mu]{F_y(X_0,\dots,X_{k-1})}-\Ex{F_y(\shat X_0,\dots,\shat X_{k-1})}\big|d\mu(y)\\
			&\le \int C\sum_{r=0}^{k-2}(\ell_{r+1}-u_r)^{-\gamma}\biggl(\norminf{F_y} +\Lip(F_y)\sum_{i=0}^{k-1}\sum_{j=0}^{u_i-\ell_i}\sdsemi{\varPhi_i}{j}\,\biggr) d\mu(y)\\
			&\le C\sum_{r=0}^{k-2}(\ell_{r+1}-u_r)^{-\gamma}\biggl(\norminf{F}+\Lip(F)\sum_{i=0}^{k-1}\sum_{j=0}^{u_i-\ell_i}\sdsemi{\varPhi_i}{j}\biggr).
		\end{split}
	\end{equation*}
	Now
	\begin{align*}
		I_2&=\Ex[\mu]{F(X_0,\dots,X_{k})}-\int_M \Ex[\mu]{F\big(X_0,\dots,X_{k-1},X_{k}(y)\big)} d\mu(y)\\
		&=\int_M F\big(X_0(x),\dots,X_{k}(x)\big)d\mu(x)-\int_{M^2} F\big(X_0(x),\dots,X_{k-1}(x),X_{k}(y)\big) d\mu^2(x,y).
	\end{align*}
	Write
	\begin{align*}
		F(X_0&(x),\dots,X_k(x))\\
		&=F(\varPhi_0(\tf^{\ell_0}x,\dots,\tf^{u_0}x);\varPhi_1(\tf^{\ell_1}x,\dots,\tf^{u_1}x);\dots;\varPhi_k(\tf^{\ell_k}x,\dots,\tf^{u_k}x))\\
		&=G(\tf^{\ell_0}x,\dots,\tf^{u_0}x;\tf^{\ell_1}x,\dots,\tf^{u_1}x;\dots;\tf^{\ell_k}x,\dots,\tf^{u_k}x).
	\end{align*}
	and 
	\begin{align*}
		F(X_0&(x),\dots,X_{k-1}(x),X_k(y))\\
		&=G(\tf^{\ell_0}x,\dots,\tf^{u_0}x;\tf^{\ell_1}x,\dots,\tf^{u_1}x;\dots;\tf^{\ell_{k-1}}x,\dots,\tf^{u_{k-1}}x;\tf^{\ell_k}y,\dots,\tf^{u_k}y)
	\end{align*}
	where $ G\map{\ms^s}{\R} $, $ s=\sum_{i=0}^k (u_i-\ell_i+1)$. By a straightforward calculation, $ G\in \sepdholsp{s}(\ms) $ and
	\[ \sum_{i=0}^{s-1}\sdsemi{G}{i}\le \sum_{i=0}^{k}\sum_{j=0}^{u_i-\ell_i}\Lip(F)\sdsemi{\varPhi_i}{j}.\]
	Hence by the Functional Correlation Bound,
	\begin{align*}
		|I_2|&=\bigg|\int_M G(\tf^{\ell_0}x,\dots,\tf^{u_0}x;\dots;\tf^{\ell_{k}}x,\dots,\tf^{u_{k}}x)d\mu(x)\\
		&\qquad-\int_{M^2} G(\tf^{\ell_0}x,\dots,\tf^{u_0}x;\dots;\tf^{\ell_{k-1}}x,\dots,\tf^{u_{k-1}}x;\tf^{\ell_{k}}y,\dots,\tf^{u_{k}}y)d\mu^2(x,y)\bigg|\\
		&\le C(\ell_{k}-u_{k-1})^{-\gamma}\biggl(\norminf{F}+\sum_{i=0}^{k} \sum_{j=0}^{u_i-\ell_i}\Lip(F)\sdsemi{\varPhi_i}{j}\biggr).
	\end{align*}
	This completes the proof.
\end{proof}
\section{Moment bounds}\label{section:moment_bounds}
	In this section we prove  Theorem~\ref{thm:moment_bd}. Throughout this section we fix $ \gamma>1$ and assume that $ \tf\map{\ms}{\ms} $ satisfies the Functional Correlation Bound with rate $ n^{-\gamma}. $

In both parts of Theorem~\ref{thm:moment_bd} we use the following moment bounds for independent, mean zero random variables, which are due to von Bahr, Esseen \cite{von1965inequalities} and Rosenthal \cite{rosenthal1970subspaces}, respectively:
\begin{lemma}\label{lemma:rosenthal}
	Fix $p\ge 1$. There exists a constant $C>0$ such that for all $k\ge 1$, for all independent, mean zero random variables $\widehat{X}_0,\dots,\widehat{X}_{k-1}\in L^p$: 
	\begin{enumerate}
		\item\label{item:von_bahr_esseen} If $1\le p\le 2$, then
		\begin{equation*}
			\Ex{\biggl|\sum_{i=0}^{k-1} \widehat{X}_i \biggr|^p}\le C \sum_{i=0}^{k-1} \Ex{|\widehat{X}_i|^p}.
		\end{equation*}
		\item\label{item:rosenthal} If $p>2$, then
	\end{enumerate}
	\begin{equation*}
		\Ex{\biggl|\sum_{i=0}^{k-1} \widehat{X}_i \biggr|^p}\le C\left(\biggl(\sum_{i=0}^{k-1} \Ex{\widehat{X}_i^2}\biggr)^{p/2}+ \sum_{i=0}^{k-1} \Ex{|\widehat{X}_i|^p}\right).
	\end{equation*}

	\vspace{-2em}\qed
\end{lemma}
Let $v,w\in \dholsp(\ms)$ be mean zero. For $ b\ge a\ge 0 $ we denote
\begin{equation*}
	S_v(a,b)=\sum_{a\le i<b}v\circ \tf^i,\quad \bbS_{v,w}(a,b)=\sum_{a\le i<j<b}v\circ \tf^i w\circ \tf^j.
\end{equation*}
Note that $S_v(n)=S_v(0,n)$ and $ \bbS_{v,w}(n)=\bbS_{v,w}(0,n)$. Some straightforward algebra yields the following proposition.
\begin{prop}\label{prop:chen_relation}
	Fix $ \ell\ge 1 $ and $ 0=a_0\le a_1\le \dots \le a_\ell.$ Then,
	\begin{enumerate}
		\item\label{item:S_n_chen} $\displaystyle S_v(a_\ell)=\sum_{i=0}^{\ell-1}S_v(a_i,a_{i+1}). $
		\item\label{item:bbS_n_chen} $ \displaystyle  
		\bbS_{v,w}(a_\ell)=\sum_{i=0}^{\ell-1}\bbS_{v,w}(a_i,a_{i+1})+\sum_{0\le i<j<\ell}S_v(a_i,a_{i+1})S_w(a_j,a_{j+1}).$\qed
	\end{enumerate}
\end{prop}
We also need the following elementary proposition:
\begin{prop}\label{prop:p_norm_diff_mvt}
	Fix $R>0$, $p\ge 1$ and an integer $k\ge 1$. Define $F\map{[-R,R]^k}{\R}$ by
	$F(y_0,\dots,y_{k-1})=|y_0+\dots+y_{k-1}|^p.$
	Then $ \norminf{F}\le (kR)^p$ and $\Lip(F)\le p(kR)^{p-1}.$
\end{prop}
\begin{proof}
	Note that $ \norminf{F}\le (kR)^p$. Fix $y=(y_0,\dots,y_{k-1}),y'=(y'_0,\dots,y'_{k-1})\in [-R,R]^k$ and set $a=|y_0+\dots+y_{k-1}|, b=|y'_0+\dots+y'_{k-1}|$. By the Mean Value Theorem,
	\begin{align*}
		|F(y_0,\dots,y_{k-1})-F(y'_0,\dots,y'_{k-1})|&=|a^p-b^p|\\
		&\le p\max\{a^{p-1},b^{p-1}\}|a-b|\\
		&\le p(kR)^{p-1}\sum_{i=0}^{k-1} |y_i-y'_i|=p(kR)^{p-1}|y-y'|,
	\end{align*}
	so $\Lip(F)\le p(kR)^{p-1}$.
\end{proof}
Let $k\ge 1,n\ge 2k $ and define $ a_i=[\tfrac{in}{2k}]$ for $ 0\le i\le 2k.$ Note that 
\begin{equation}\label{eq:a_i_diff_bds}
	\tfrac{n}{2k}-1\le a_{i+1}-a_i\le \tfrac{n}{2k}+1\le \tfrac{n}{k}.
\end{equation}
For $ 0\le i< k $ let $ X_i=S_v(a_{2i},a_{2i+1}). $ Let $\widehat{X}_0,\dots,\widehat{X}_{k-1}$ be independent random variables with $\widehat{X}_i\eqd X_i$.
\begin{lemma}\label{lemma:weak_dep_moment_bd}
	There exists a constant $ C>0 $ such that
	\begin{equation*}
		\Ex[\mu]{\bigg|\sum_{i=0}^{k-1}X_i\bigg|^{2\gamma}}\le Ck^{1+\gamma} n^\gamma\dhnorm{v}^{2\gamma}+\Ex{\bigg|\sum_{i=0}^{k-1}\widehat{X}_i\bigg|^{2\gamma}},
	\end{equation*}
	for all $n\ge 2k,k\ge 1$, for any $v\in \dholsp(\ms)$.
\end{lemma}
\begin{proof}
	Note that
	\begin{align*}
		X_i(x)=\sum_{q=a_{2i}}^{a_{2i+1}-1}v(\tf^q x)=\varPhi_i(\tf^{\ell_i}x,\dots,\tf^{u_i}x),
	\end{align*}
	where $\ell_i=a_{2i}, u_i=a_{2i+1}-1$ and
	\begin{equation*}
		\varPhi_i(x_0,\dots,x_{u_i-\ell_i})=\sum_{j=0}^{u_i-\ell_i}v(x_j).
	\end{equation*}
	Let $ R=\max_i \norminf{\varPhi_i}. $ Then 
	\[  	\Ex[\mu]{\bigg|\sum_{i=0}^{k-1}X_i\bigg|^{2\gamma}}=\Ex[\mu]{F(X_0,\dots,X_{k-1})} \]
	where $F\map{[-R,R]^k}{\R}$ is given by $ F(y_0,\dots,y_{k-1})=|y_0+\dots+y_{k-1}|^{2\gamma}.$ Hence by Lemma~\ref{lemma:fcb_weak_dep},
	\begin{align*}
		\Ex[\mu]{\bigg|\sum_{i=0}^{k-1}X_i\bigg|^{2\gamma}}&\le A+\Ex{\bigg|\sum_{i=0}^{k-1}\shat{X}_i\bigg|^{2\gamma}}\label{eq:moment_bd_I_1_splitting}
	\end{align*}
	where
	\begin{equation}\label{eq:moment_bd_overall_fcb_error}
		|A|\le C\sum_{r=0}^{k-2}(\ell_{r+1}-u_r)^{-\gamma}\biggl(\norminf{F}+\Lip(F)\sum_{i=0}^{k-1}\sum_{j=0}^{u_i-\ell_i}\sdsemi{\varPhi_i}{j}\biggr).
	\end{equation}
	
	It remains to bound $ A $. First we bound the expressions $ \sdsemi{\varPhi_i}{j} $.
	Fix $ 0\le i< k $ and $0\le j\le u_i-\ell_i$. For $x_0,\dots,x_{k-1}, x'_j\in \ms$, 
	\begin{align*}
		|\varPhi_i(x_0,\dots,x_{u_i-\ell_i})-\varPhi_i(x_0,\dots,x_{j-1},x'_j,x_{j+1}\dots,x_{u_i-\ell_i})|&=|v(x_j)-v(x'_j)|
	\end{align*}
	so $ \sdsemi{\varPhi_i}{j}\le \dsemi{v}.$ Note that by~\eqref{eq:a_i_diff_bds},
	$ \norminf{\varPhi_i}\le (a_{2i+1}-a_{2i})\norminf{v}\le \tfrac{n}{k}\norminf{v}. $
	Hence by Proposition~\ref{prop:p_norm_diff_mvt}, 
	\begin{equation}\label{eq:moment_bd_norminf_F}
		\norminf{F}\le 2\gamma(n\norminf{v})^{2\gamma}
	\end{equation}
	and $\Lip(F)\le 2\gamma(n\norminf{v})^{2\gamma-1}.$ 
	
	Thus
	\begin{align}
		\Lip(F)\sum_{i=0}^{k-1}\sum_{j=0}^{u_i-\ell_i}\sdsemi{\varPhi_i}{j}\le & 2\gamma(n\norminf{v})^{2\gamma-1}\sum_{i=0}^{k-1}\sum_{j=0}^{u_i-\ell_i}\sdsemi{\varPhi_i}{j}\nonumber\\
		\le& 2\gamma(n\norminf{v})^{2\gamma-1}\sum_{i=0}^{k-1}(u_i-\ell_i+1)\dsemi{v}\nonumber\\
		\le& 2\gamma(n\norminf{v})^{2\gamma-1}n\dsemi{v}.\label{eq:moment_bd_Lip_F_sum}
	\end{align}
	Now by \eqref{eq:a_i_diff_bds},
	$\ell_{r+1}-u_r=a_{2r+2}-(a_{2r+1}-1)\ge \tfrac{n}{2k}$ for each $0\le r\le k-2$. Hence
	\begin{equation}\label{eq:gap_size_sum_bd}
		\sum_{r=0}^{k-2} (\ell_{r+1}-u_r)^{-\gamma}\le k(\tfrac{n}{2k})^{-\gamma}=2^{\gamma}k^{1+\gamma}n^{-\gamma}.
	\end{equation}
	Substituting \eqref{eq:moment_bd_norminf_F}, \eqref{eq:moment_bd_Lip_F_sum} and \eqref{eq:gap_size_sum_bd} into \eqref{eq:moment_bd_overall_fcb_error} gives
	\begin{align*}
		|A|&\le 2^{\gamma}k^{1+\gamma}n^{-\gamma}(2\gamma(n\norminf{v})^{2\gamma} + 2\gamma(n\norminf{v})^{2\gamma-1}n\dsemi{v})\\
		&\le 2^{1+\gamma}\gamma Ck^{1+\gamma}n^{-\gamma}  \, (n\dhnorm{v})^{2\gamma}= 2^{1+\gamma}\gamma Ck^{1+\gamma}n^\gamma \dhnorm{v}^{2\gamma},
	\end{align*}
	as required.
\end{proof}
We are now ready to prove the moment bound for $S_v(n)$ (Theorem~\ref{thm:moment_bd}\ref{item:moment_bd}).
\begin{proof}[Proof of Theorem~\ref{thm:moment_bd}\ref{item:moment_bd}]
	We prove by induction that there exists $ D> 0 $ such that
	\begin{equation}\label{eq:moment_bd_induct}
		\pnorm{S_v(m)}_{2\gamma}\le Dm^{1/2}\dhnorm{v}
	\end{equation}
	for all $ m\ge 1$, for any mean zero $v\in \dholsp(\ms)$.

	\textbf{\underline{Claim.}} There exists $ C>0 $ such that for all mean zero $ v\in \dholsp(\ms) $, for any $D>0$, for any $ k\ge 1 $ and any $ n\ge 2k $ such that~\eqref{eq:moment_bd_induct} holds for all $m<n$, we have
	\begin{equation*}
		\pnorm{S_v(n)}_{2\gamma}^{2\gamma}\le C(k^{1+\gamma}+k^{1-\gamma}D^{2\gamma})n^\gamma\dhnorm{v}^{2\gamma}.
	\end{equation*}
	
	Now fix $ k\ge 1 $ such that $ Ck^{1-\gamma}\le \frac{1}{2} $. Fix $ D>0 $ such that $ Ck^{1+\gamma}\le \frac{1}{2}D^{2\gamma} $ and~\eqref{eq:moment_bd_induct} holds for all $ m<2k $ and any mean zero $ v\in \dholsp(\ms) $. Then the claim shows that for any $ n\ge 2k $ such that~\eqref{eq:moment_bd_induct} holds for all $ m<n $, we have $ \pnorm{S_v(n)}_{2\gamma}^{2\gamma}\le D^{2\gamma}n^\gamma \dhnorm{v}^{2\gamma}.$ Hence by induction, \eqref{eq:moment_bd_induct} holds for all $ m\ge 1. $
	
	It remains to prove the claim. Note that in the following the constant $ C>0 $ may vary from line to line.
	
	Fix $ n\ge 2k $ and assume that~\eqref{eq:moment_bd_induct} holds for all $ m<n $. 
	By Proposition~\ref{prop:chen_relation}\ref*{item:S_n_chen},
	\[S_v(n)=\sum_{i=0}^{2k-1}S_v(a_i,a_{i+1})=I_1+I_2,\]
	where
	\begin{equation*}
		I_1=\sum_{i=0}^{k-1}S_v(a_{2i},a_{2i+1}),\quad I_2=\sum_{i=0}^{k-1}S_v(a_{2i+1},a_{2i+2}).
	\end{equation*}
	
	We first bound $ \pnorm{I_1}_{2\gamma}$. Write $X_i=S_v(a_{2i},a_{2i+1})$ so that $I_1=\sum_{i=0}^{k-1} X_i$. By Lemma~\ref{lemma:weak_dep_moment_bd},
	\begin{equation}\label{eq:moment_bd_initial_I_1_bd}
		\pnorm{I_1}_{2\gamma}^{2\gamma}=\Ex[\mu]{\bigg|\sum_{i=0}^{k-1}X_i\bigg|^{2\gamma}}\le Ck^{1+\gamma} n^\gamma\dhnorm{v}^{2\gamma}+\Ex{\bigg|\sum_{i=0}^{k-1}\widehat{X}_i\bigg|^{2\gamma}}. 	
	\end{equation}
	We now bound $ \Ex{|\sum_{i=0}^{k-1}\shat{X}_i|^{2\gamma}} $ by using Lemma~\ref{lemma:rosenthal} and the inductive hypothesis.
	
	Fix $0\le i<k$. By stationarity, $X_i=S_v(a_{2i},a_{2i+1})\eqd S_v(a_{2i+1}-a_{2i}).$
	Thus by the inductive hypothesis~\eqref{eq:moment_bd_induct}, $ \Ex[\mu]{|X_i|^{2\gamma}}\le D^{2\gamma}(a_{2i+1}-a_{2i})^{\gamma}\dhnorm{v}^{2\gamma}$.
	Hence by~\eqref{eq:a_i_diff_bds},
	\begin{align*}
		\sum_{i=0}^{k-1} \Ex{|\shat{X}_i|^{2\gamma}}&\le \sum_{i=0}^{k-1}D^{2\gamma}(a_{2i+1}-a_{2i})^{\gamma}\dhnorm{v}^{2\gamma}\nonumber\\
		&\le \sum_{i=0}^{k-1} D^{2\gamma}(n/k)^\gamma\dhnorm{v}^{2\gamma}=D^{2\gamma}k^{1-\gamma}n^\gamma\dhnorm{v}^{2\gamma}.
	\end{align*}
	Now by the \fcb{}, $ |\Ex[\mu]{v\, v\circ T^n}|\le Cn^{-\gamma}\dhnorm{v}^2$. By a standard calculation, it follows that $ \Ex[\mu]{S_v(n)^2}\le Cn\dhnorm{v}^2$. Thus
	\begin{align*}
		\sum_{i=0}^{k-1} \Ex{\widehat{X}_i^2}&= \sum_{i=0}^{k-1}\Ex[\mu]{S_v(a_{2i+1}-a_{2i})^2}\\
		&\le \sum_{i=0}^{k-1} C(a_{2i+1}-a_{2i})\dhnorm{v}^2
		\le C(a_{2k-1}-a_0)\dhnorm{v}^2\\
		&\le Cn\dhnorm{v}^2.
	\end{align*}
	By Lemma~\ref{lemma:rosenthal}\ref*{item:rosenthal}, it follows that
	\begin{align*}\label{eq:moment_bd_overall_indep_moment_bd}
		\Ex{\biggl|\sum_{i=0}^{k-1} \widehat{X}_i \biggr|^{2\gamma}}&\le C\big((C n\dhnorm{v}^2)^\gamma+D^{2\gamma}k^{1-\gamma}n^\gamma\dhnorm{v}^{2\gamma}\big)\\
		&\le C(1+D^{2\gamma}k^{1-\gamma})n^\gamma\dhnorm{v}^{2\gamma}.
	\end{align*}
	Hence by~\eqref{eq:moment_bd_initial_I_1_bd}, overall
	\[\pnorm{I_1}_{2\gamma}^{2\gamma}\le C(k^{1+\gamma}+D^{2\gamma}k^{1-\gamma})n^\gamma\dhnorm{v}^{2\gamma}.\]
	Exactly the same argument applies to $|I_2|_{2\gamma}^{2\gamma}. $ The conclusion of the claim follows by noting that
	\[ \pnorm{S_v(n)}_{2\gamma}^{2\gamma}=\pnorm{I_1+I_2}_{2\gamma}^{2\gamma}\le 2^{2\gamma}(\pnorm{I_1}_{2\gamma}+\pnorm{I_2}_{2\gamma}).\qedhere \]
\end{proof}
We now prove Theorem~\ref{thm:moment_bd}\ref*{item:iterated_moment_bd}. Our proof follows the same lines as that of part~\ref{item:moment_bd}. 

Let $n, k\ge 1. $ Recall that $ a_i=\left[\tfrac{in}{2k}\right] $. For $ 0\le i< k $ define mean zero random variables $ X_i $ on $ (\ms,\mu) $ by
\[X_i=\bbS_{v,w}(a_{2i},a_{2i+1})-\Ex[\mu]{\bbS_{v,w}(a_{2i},a_{2i+1})}. \]
Let $\widehat{X}_0,\dots,\widehat{X}_{k-1}$ be independent random variables with $\widehat{X}_i\eqd X_i$.

The following lemma plays the same role that Lemma~\ref{lemma:weak_dep_moment_bd} played in the proof of Theorem~\ref{thm:moment_bd}\ref*{item:moment_bd}. 
\begin{lemma}\label{lemma:iterated_weak_dep_moment_bd}
	There exists a constant $ C>0 $ such that for any $v,w\in \dholsp(\ms)$,
	\begin{equation*}
		\Ex[\mu]{\bigg|\sum_{i=0}^{k-1}X_i\bigg|^{\gamma}}\le Ck n^\gamma\dhnorm{v}^{\gamma}\dhnorm{w}^\gamma+\Ex{\bigg|\sum_{i=0}^{k-1}\widehat{X}_i\bigg|^{\gamma}}
	\end{equation*}	
	for all $n\ge 2k,k\ge 1$.
\end{lemma}
\begin{proof}
	Note that
	\begin{align*}
		X_i(x)&=\sum_{a_{2i}\le q<r\le a_{2i+1}-1}v(\tf^q x)w(\tf^r x)-\Ex[\mu]{\bbS_{v,w}(a_{2i},a_{2i+1})}\\
		&=\varPhi_i(\tf^{\ell_i}x,\dots,\tf^{u_i}x),
	\end{align*}
	where $ \ell_i=a_{2i}, u_i=a_{2i+1}-1 $ and 
	\begin{equation*}
		\varPhi_i(x_0,\dots,x_{u_i-\ell_i})= \sum_{0\le q<r\le u_i-\ell_i}v(x_q) w(x_r)-\Ex[\mu]{\bbS_{v,w}(a_{2i},a_{2i+1})}.
	\end{equation*}
	Let $ R=\max_i \norminf{\varPhi_i} $. Observe that  
	\[ 	\Ex[\mu]{\bigg|\sum_{i=0}^{k-1}X_i\bigg|^\gamma}=\Ex[\mu]{F(X_0,\dots,X_{k-1})}, \]
	where $F\map{[-R,R]^k}{\R}$ is given by $ F(y_0,\dots,y_{k-1})=|y_0+\dots+y_{k-1}|^{\gamma}.$ Hence by Lemma~\ref{lemma:fcb_weak_dep},
	\begin{align*}
		\Ex[\mu]{\bigg|\sum_{i=0}^{k-1}X_i\bigg|^{\gamma}}&\le A+\Ex{\bigg|\sum_{i=0}^{k-1}\widehat{X}_i\bigg|^{2\gamma}}
	\end{align*}
	where
	\begin{equation}\label{eq:iterated_moment_bd_overall_fcb_error}
		|A|\le C\sum_{r=0}^{k-2}(\ell_{r+1}-u_r)^{-\gamma}\biggl(\norminf{F}+\Lip(F)\sum_{i=0}^{k-1}\sum_{j=0}^{u_i-\ell_i}\sdsemi{\varPhi_i}{j}\biggr).
	\end{equation}
	It remains to bound $ A $. The first step is to bound the expressions $ \sdsemi{\varPhi_i}{j}$. Fix $ 0\le i<k, 0\le j\le u_i-\ell_i.$ Let $ x_0,\dots,x_{k-1},x_j'\in \ms.$ Note that 
	\[\varPhi_i(x_0,\dots,x_{u_i-\ell_i})-\varPhi_i(x_0,\dots,x_{j-1},x'_j,x_{j+1}\dots,x_{u_i-\ell_i})=J_1+J_2,  \]
	where
	\begin{align*}
		J_1&=\sum_{j<r\le u_i-\ell_i}(v(x_j)w(x_r)-v(x'_j)w(x_r)),\qquad
		J_2=\sum_{0\le q<j}(v(x_q)w(x_j)-v(x_q)w(x'_j)).
	\end{align*}
	Now,
	\[ |J_1|\le \sum_{j<r\le u_i-\ell_i}|v(x_j)-v(x'_j)||w(x_r)|\le \norminf{w}\sum_{j<r\le u_i-\ell_i}|v(x_j)-v(x'_j)| \]
	and similarly $ |J_2|\le\norminf{v}\sum_{0\le q<j}|w(x_j)-w(x'_j)| $, so
	\begin{equation*}
		\sdsemi{\varPhi_i}{j}\le (u_i-\ell_i)\dhnorm{v}\dhnorm{w}.
	\end{equation*}
	Now recall from~\eqref{eq:a_i_diff_bds} that $ u_i-\ell_i+1=a_{2i+1}-a_{2i}\le n/k $ so
	\begin{equation}\label{eq:iterated_mom_holder_coeff_sum}
		\sum_{i=0}^{k-1}\sum_{j=0}^{u_i-\ell_i}\sdsemi{\varPhi_i}{j}\le \sum_{i=0}^{k-1}(u_i-\ell_i+1)^2 \dhnorm{v}\dhnorm{w}\le \tfrac{n^2}{k}\dhnorm{v}\dhnorm{w}.
	\end{equation}
	Next note that
	\begin{align*}
		\norminf{\varPhi_i}&\le \sum_{0\le q<r\le u_i-\ell_i}\norminf{v}\norminf{w}+\norminf{\bbS_{v,w}(a_{2i},a_{2i+1})}\\
		&\le 2(n/k)^2\norminf{v}\norminf{w}
	\end{align*}
	so by Proposition~\ref{prop:p_norm_diff_mvt}, 
$\norminf{F}\le \bigl(\tfrac{2n^2}{k}\norminf{v}\norminf{w}\bigr)^\gamma$
and
	$ \Lip(F)\le \gamma\big(\tfrac{2n^2}{k} \norminf{v}\norminf{w}\big)^{\gamma-1}. $ Combining these bounds with~\eqref{eq:gap_size_sum_bd},~\eqref{eq:iterated_moment_bd_overall_fcb_error}  and~\eqref{eq:iterated_mom_holder_coeff_sum} yields that
	\begin{align*}
		|A|&\le C2^{\gamma}k^{1+\gamma}n^{-\gamma}\bigl((\tfrac{2n^2}{k}\norminf{v}\norminf{w})^\gamma+\gamma(\tfrac{2n^2}{k} \norminf{v}\norminf{w}\big)^{\gamma-1}\tfrac{n^2}{k}\dhnorm{v}\dhnorm{w}\bigr)\\
		&\le 2^{2\gamma}(1+\gamma/2) Ckn^\gamma\dhnorm{v}^\gamma \dhnorm{w}^\gamma,
	\end{align*}
	as required.
\end{proof}
We are now ready to prove Theorem~\ref{thm:moment_bd}\ref*{item:iterated_moment_bd}.
\begin{proof}[Proof of Theorem~\ref{thm:moment_bd}\ref*{item:iterated_moment_bd}]
	We prove by induction that there exists $ D> 0 $ such that
	\begin{equation}\label{eq:iterated_moment_bd_induct}
		\pnorm{\bbS_{v,w}(m)}_{\gamma}\le Dm\dhnorm{v}\dhnorm{w}
	\end{equation}
	for all $ m\ge 1$, for any $v,w \in \dholsp(\ms)$ mean zero.
	
	\textbf{\underline{Claim.}} There exists $ C>0 $ such that for all $ v,w\in \dholsp(\ms) $ mean zero, for any $ D>0$, any $k\ge 1$ and any $ n\ge 2k $ such that~$ \eqref{eq:iterated_moment_bd_induct} $ holds for all $m<n$, we have
	\begin{equation*}
		\pnorm{\bbS_{v,w}(n)}_{\gamma}^\gamma\le C(k^\gamma+(k^{1-\gamma}+k^{-\gamma/2})D^\gamma)(n \dhnorm{v}\dhnorm{w})^\gamma.
	\end{equation*}
	
	Now fix $ k\ge 1 $ such that $ C(k^{1-\gamma}+k^{-\gamma/2})\le \frac{1}{2} $. Fix $ D>0 $ such that ${Ck^\gamma\le \frac{1}{2}D^{\gamma}}$ and~\eqref{eq:iterated_moment_bd_induct} holds for all $ m<2k $ and any mean zero $ v,w\in \dholsp(\ms)$. Then the claim shows that if $ n\ge 2k $ and~\eqref{eq:iterated_moment_bd_induct} holds for all $ m<n $, then  $ \pnorm{\bbS_{v,w}(n)}_{\gamma}^\gamma\le D^\gamma(n\dhnorm{v}\dhnorm{w})^\gamma $. Hence by induction, \eqref{eq:iterated_moment_bd_induct} holds for all $ m\ge 1. $
	
	It remains to prove the claim. Note that in the following the constant $ C>0 $ may vary from line to line.
	
	Fix $ n\ge 2k $ and assume that~\eqref{eq:iterated_moment_bd_induct} holds for all $m<n$. Recall that $ a_i= \big[\tfrac{in}{2k}\big]$ for $ 0\le i\le 2k. $ By Proposition~\ref{prop:chen_relation}\ref*{item:bbS_n_chen}, \begin{equation*}
		\bbS_{v,w}(n)=\sum_{0\le i<j<2k}S_v(a_i,a_{i+1})S_w(a_j,a_{j+1})+\sum_{i=0}^{2k-1}\bbS_{v,w}(a_i,a_{i+1})=I_1+I_2+I_3+I_4,
	\end{equation*}
	where
	\begin{align*}
		I_1&=\sum_{0\le i<j<2k}S_v(a_i,a_{i+1})S_w(a_j,a_{j+1}),\quad
		I_2=\sum_{i=0}^{2k-1} \Ex[\mu]{\bbS_{v,w}(a_i,a_{i+1})},\\
		I_3&=\sum_{i=0}^{k-1}\big(\bbS_{v,w}(a_{2i},a_{2i+1})-\Ex[\mu]{\bbS_{v,w}(a_{2i},a_{2i+1})}\big),\\ I_4&=\sum_{i=0}^{k-1}\big(\bbS_{v,w}(a_{2i+1},a_{2i+2})-\Ex[\mu]{\bbS_{v,w}(a_{2i+1},a_{2i+2})}\big).
	\end{align*}
	Recall from~\eqref{eq:a_i_diff_bds} that $ a_{i+1}-a_i\le n/k $. Hence by Theorem~\ref{thm:moment_bd}\ref*{item:moment_bd},
	\begin{align*}
		\pnorm{I_1}_\gamma&\le \sum_{0\le i<j<2k}\pnorm{S_v(a_i,a_{i+1})S_w(a_j,a_{j+1})}_{\gamma}\\
		&\le \sum_{0\le i<j<2k}\pnorm{S_v(a_i,a_{i+1})}_{2\gamma}\pnorm{S_w(a_j,a_{j+1})}_{2\gamma}\\
		&\le \sum_{0\le i<j<2k}C^2(a_{i+1}-a_i)^{1/2}\dhnorm{v}(a_{j+1}-a_j)^{1/2}\dhnorm{w}\\
		& \le \sum_{0\le i<j<2k}C^2(n/k)^{1/2}\dhnorm{v}(n/k)^{1/2}\dhnorm{w}\le Ckn\dhnorm{v}\norm{w}.
	\end{align*}
	Now by the \fcb{}, $ |\Ex[\mu]{v\, w\circ T^n}|\le Cn^{-\gamma}\dhnorm{v}\dhnorm{w}$. By a standard calculation, it follows that $ |\Ex[\mu]{\bbS_{v,w}(n)}\!|\le Cn\dhnorm{v}\dhnorm{w}$. Thus
	\begin{align*}
		|I_2|&\le \sum_{i=0}^{2k-1} |\Ex[\mu]{\bbS_{v,w}(a_i,a_{i+1})}|\\
		&\le \sum_{i=0}^{2k-1} C(a_{i+1}-a_{i})\dhnorm{v}\dhnorm{w}= C(a_{2k}-a_0)\dhnorm{v}\dhnorm{w}\\
		&=Cn\dhnorm{v}\dhnorm{w}.
	\end{align*}
	We now bound $ \pnorm{I_3}_\gamma^\gamma. $ Note that $ I_3=\sum_{i=0}^{k-1}X_i, $ where 
	$ X_i=\bbS_{v,w}(a_{2i},a_{2i+1})-\Ex[\mu]{\bbS_{v,w}(a_{2i},a_{2i+1})}.$
	Hence by Lemma~\ref{lemma:iterated_weak_dep_moment_bd},
	\begin{equation}\label{eq:iterated_I_3_bd}
		\pnorm{I_3}_\gamma^\gamma=\Ex[\mu]{\bigg|\sum_{i=0}^{k-1}X_i\bigg|^{\gamma}}\le Ckn^\gamma\dhnorm{v}^\gamma \dhnorm{w}^\gamma+\Ex{\bigg|\sum_{i=0}^{k-1}\widehat{X}_i\bigg|^{\gamma}}.
	\end{equation}
	Fix $ 0\le i<k. $ By stationarity,
	$ X_i\eqd \bbS_{v,w}(a_{2i+1}-a_{2i})-\Ex[\mu]{\bbS_{v,w}(a_{2i+1}-a_{2i})}. $
	Now by the inductive hypothesis~\eqref{eq:iterated_moment_bd_induct}, $ \pnorm{\bbS_{v,w}(a_{2i+1}-a_{2i})}_\gamma\le  D(a_{2i+1}-a_{2i})\dhnorm{v}\dhnorm{w}, $ so
	\begin{align*}
		\pnorm{X_i}_\gamma&\le \pnorm{\bbS_{v,w}(a_{2i+1}-a_{2i})}_\gamma+|\Ex[\mu]{\bbS_{v,w}(a_{2i+1}-a_{2i})}|\\
		&\le 2D(a_{2i+1}-a_{2i})\dhnorm{v} \dhnorm{w}.
	\end{align*}
	It follows that 
	\begin{align*}
		\sum_{i=0}^{k-1}\Ex{|\widehat X_i|^{\gamma}}&\le \sum_{i=0}^{k-1}2^\gamma D^\gamma(a_{2i+1}-a_{2i})^\gamma(\dhnorm{v} \dhnorm{w})^\gamma\\
		&\le \sum_{i=0}^{k-1}2^\gamma D^\gamma(n/k)^\gamma(\dhnorm{v} \dhnorm{w})^\gamma=2^\gamma D^\gamma k^{1-\gamma}(n\dhnorm{v}\dhnorm{w})^\gamma.
	\end{align*}
	If $ 1< \gamma\le 2 $, then by Lemma~\ref{lemma:rosenthal}\ref*{item:von_bahr_esseen},
	\begin{equation*}
		\Ex{\bigg|\sum_{i=0}^{k-1}\widehat X_i\bigg|^{\gamma}}\le 2^\gamma CD^\gamma k^{1-\gamma}(n\dhnorm{v}\dhnorm{w})^\gamma.
	\end{equation*}
	Suppose on the other hand that $ \gamma>2. $ Note that \[ \pnorm{\shat X_i}_2\le \pnorm{\shat X_i}_\gamma\le 2D{(a_{2i+1}-a_{2i})}\dhnorm{v} \dhnorm{w}\] so 
	\begin{align*}
		\sum_{i=0}^{k-1}\Ex{\widehat X_i^2}&\le \sum_{i=0}^{k-1}4D^2 (a_{2i+1}-a_{2i})^2(\dhnorm{v}\dhnorm{w})^2\\
		&\le \sum_{i=0}^{k-1}4D^2(n/k)^2 (\dhnorm{v}\dhnorm{w})^2=4D^2k^{-1} (n\dhnorm{v}\dhnorm{w})^2.
	\end{align*}
	Hence by Lemma~\ref{lemma:rosenthal}\ref*{item:rosenthal},
	\begin{align*}
		\Ex{\bigg|\sum_{i=0}^{k-1}\widehat X_i\bigg|^{\gamma}}&\le C\bigg(\big(4D^2 k^{-1}(n\dhnorm{v}\dhnorm{w})^2\big)^{\gamma/2}+2^\gamma D^\gamma k^{1-\gamma}\big(n\dhnorm{v}\dhnorm{w}\big)^\gamma\bigg)\\
		&=2^\gamma CD^\gamma(k^{-\gamma/2}+k^{1-\gamma})(n\dhnorm{v} \dhnorm{w})^\gamma.
	\end{align*}
	Hence for any $ \gamma>1 $,
	\begin{equation*}
		\Ex{\bigg|\sum_{i=0}^{k-1}\widehat X_i\bigg|^{\gamma}}\le CD^\gamma(k^{-\gamma/2}+k^{1-\gamma})(n\dhnorm{v} \dhnorm{w})^\gamma.
	\end{equation*}
	By~\eqref{eq:iterated_I_3_bd}, it follows that \[ \pnorm{I_3}_\gamma^\gamma\le  C\big(k+D^\gamma(k^{-\gamma/2}+k^{1-\gamma})\big)(n\dhnorm{v} \dhnorm{w})^\gamma. \] Exactly the same argument applies to $ \pnorm{I_4}^\gamma_\gamma.$ The conclusion of the claim follows by noting that
	\begin{align*}
		\pnorm{\bbS_{v,w}(n)}_\gamma^\gamma&=\pnorm{I_1+I_2+I_3+I_4}_\gamma^\gamma\le 4^\gamma(\pnorm{I_1}_\gamma^\gamma+\pnorm{I_2}_\gamma^\gamma+\pnorm{I_3}_\gamma^\gamma+\pnorm{I_4}_\gamma^\gamma)\\
		&\le C\big(k^\gamma+1+2(k+D^\gamma(k^{-\gamma/2}+k^{1-\gamma})\big)(n\dhnorm{v} \dhnorm{w})^\gamma\\
		&\le C\big(k^\gamma+D^\gamma(k^{-\gamma/2}+k^{1-\gamma})\big)(n\dhnorm{v} \dhnorm{w})^\gamma,
	\end{align*}
	as required.
\end{proof}

\textbf{Acknowledgements} The author would like to thank his supervisor Ian Melbourne for suggesting the problem considered in this paper, providing constant feedback and participating in many helpful discussions. He is also grateful to the anonymous referee for their comments, which improved the presentation of this paper.
	\bibliographystyle{alpha}
	\bibliography{bibliography}
\end{document}